\documentclass[11pt,a4paper]{article}
\RequirePackage{amsmath}
\RequirePackage{amssymb,latexsym}
\RequirePackage{amsthm}
\RequirePackage{enumerate}
\RequirePackage[hang,flushmargin,bottom,stable]{footmisc}
\RequirePackage[margin=0.35in,format=hang,font=small,labelfont=bf]{caption}
\RequirePackage[normalem]{ulem}
\RequirePackage{needspace}
\RequirePackage{graphicx,tikz}
\RequirePackage[T1]{fontenc}
\RequirePackage[sc]{mathpazo}
\RequirePackage[color]{changebar}
\RequirePackage[colorlinks=true,urlcolor=blue,linkcolor=blue,citecolor=blue]{hyperref}
\usepackage[only,mapsfrom]{stmaryrd}

\cbcolor{red}
\nochangebars

\hypersetup{
pdfstartview={XYZ null null 1.00},
pdflang={en-GB},
}

\newtheorem{thmO}{Theorem}
\newtheorem{propO}[thmO]{Proposition}
\newtheorem{obsO}[thmO]{Observation}
\newtheorem{conjO}[thmO]{Conjecture}
\newtheorem{questO}[thmO]{Question}

\newcommand{\+}{\hspace{0.07 em}}
\newcommand{\AAA}{\mathcal{A}}
\newcommand{\bb}{{\boldsymbol{\beta}}}
\newcommand{\BBB}{\mathcal{B}}
\newcommand{\bbC}{\mathbb{C}}
\newcommand{\bbE}{\mathbb{E}}
\newcommand{\bbN}{\mathbb{N}}
\newcommand{\bbQ}{\mathbb{Q}}
\newcommand{\bbR}{\mathbb{R}}
\newcommand{\bbZ}{\mathbb{Z}}
\newcommand{\bg}{{\boldsymbol{\gamma}}}
\newcommand{\ccc}{\mathfrak{c}}
\newcommand{\CCC}{\mathcal{C}}
\newcommand{\ceil}[1]{\left\lceil #1 \right\rceil}
\newcommand{\chplane}[1]{ \{ s : \re s \geqs #1 \} }
\newcommand{\DDD}{\mathcal{D}}
\newcommand{\defeq}{\mathchoice{\;:=\;}{:=}{:=}{:=}}
\newcommand{\dgf}{{\small DGF}}
\newcommand{\EEE}{\mathcal{E}}
\newcommand{\eq}{\mathchoice{\;=\;}{=}{=}{=}}
\newcommand{\expecn}[2]{\bbE_{#1}\big[#2\big]}
\newcommand{\fff}{\mathfrak{f}}
\newcommand{\FFF}{\mathcal{F}}
\newcommand{\floor}[1]{\left\lfloor #1 \right\rfloor}
\newcommand{\geqs}{\geqslant}
\newcommand{\gf}{{\small GF}}
\newcommand{\GGG}{\mathcal{G}}
\newcommand{\helix}{S_\mathbf{H}}
\newcommand{\HIDE}[1]{}
\newcommand{\igf}{{\small IGF}}
\DeclareMathOperator{\im}{\mathfrak{Im}}
\newcommand{\leqs}{\leqslant}
\newcommand{\liminfty}[1][n]{\lim\limits_{#1\rightarrow\infty}}
\newcommand{\MMM}{\mathcal{M}}
\newcommand{\oeis}[1]{\href{https://oeis.org/#1}{#1}}
\newcommand{\ohplane}[1]{ \{ s : \re s > #1 \} }
\newcommand{\pity}{primitive irrationality}
\newcommand{\Pity}{Primitive irrationality}
\newcommand{\ppp}{\mathfrak{p}}
\newcommand{\PPP}{\mathcal{P}}
\newcommand{\QQQ}{\mathcal{Q}}
\newcommand{\prim}{primitive} 
\newcommand{\primi}{primitive irrational} 
\DeclareMathOperator{\re}{\mathfrak{Re}}

\newcommand{\sapprox}{\mathchoice{\;\approx\;}{\approx}{\approx}{\approx}}
\newcommand{\seq}[1]{\text{\textsc{Seq}}[#1]}
\newcommand{\sleqs}{\mathchoice{\;\leqs\;}{\leqs}{\leqs}{\leqs}}
\newcommand{\sminus}{\mathchoice{\:-\:}{-}{-}{-}}
\newcommand{\splus}{\mathchoice{\:+\:}{+}{+}{+}}
\newcommand{\ssim}{\mathchoice{\;\sim\;}{\sim}{\sim}{\sim}}
\newcommand{\sss}{\mathfrak{s}}
\newcommand{\SSS}{\mathcal{S}}
\newcommand{\STAR}{{\scalebox{1.1}{$*$}}}

\newcommand{\together}[1]{\Needspace*{#1\baselineskip}}
\newcommand{\ttt}{\mathfrak{t}}
\newcommand{\TTT}{\mathcal{T}}
\newcommand{\UUU}{\mathcal{U}}
\newcommand{\veps}{\varepsilon}
\newcommand{\vphi}{\varphi}
\newcommand{\VVV}{\mathcal{V}}
\newcommand{\WWW}{\mathcal{W}}
\newcommand{\YYY}{\mathcal{Y}}

\newcommand{\plotptradius}{0.275}
\newcommand{\setplotptradius}[1]{\renewcommand{\plotptradius}{#1}}
\newcommand{\plotpt}[3][] 
{ \fill[#1,radius=\plotptradius] (#2,#3) circle; }
\newcommand{\plotpermnobox}[3][]  
{
  \foreach \y [count=\x] in {#3}
  {
    \ifnum0=\y {} \else {
      \plotpt[#1]{\x}{\y}
    } \fi
  }
}

\newcommand{\myTitle}{\texorpdfstring{Introducing irrational enumeration: \\analytic combinatorics for objects of irrational size}{Introducing irrational enumeration: analytic combinatorics for objects of irrational size}}
\title{\textbf{\myTitle}}

\author{David Bevan${}^\dagger$ and Julien Cond\'e${}^{\,\ddag}$}

\hypersetup{
pdftitle={\myTitle},
pdfauthor={David Bevan and Julien Cond\'e},
pdfstartview={XYZ null null 1.00}
}

\date{}

\begin{document}
\maketitle

{\begin{NoHyper}
\let\thefootnote\relax\footnotetext
{${}^\dagger$Department of Mathematics and Statistics, The University of Strathclyde, Glasgow, Scotland.}
\let\thefootnote\relax\footnotetext
{${}^\ddag$Institut Denis Poisson, L'Universit\'e d'Orl\'eans, France.}
\end{NoHyper}}

{\begin{NoHyper}
\let\thefootnote\relax
\end{NoHyper}}

\begin{abstract}
\noindent
We extend the scope of analytic combinatorics to classes containing objects that have irrational sizes.
The generating function for such a class is a power series that admits
irrational exponents (which we call a Ribenboim series).
A transformation then yields a generalised Dirichlet series from which the asymptotics of the coefficients can be extracted by singularity analysis
using an appropriate Tauberian theorem.
In practice, the asymptotics can often be determined directly from the original generating function.
We illustrate the technique with a variety of applications, including
tilings with tiles of irrational area,
ordered integer factorizations,
lattice walks enumerated by Euclidean length,
and plane trees with vertices of irrational size.
We also explore phase transitions in the asymptotics of families of irrational combinatorial classes.
\end{abstract}

\section{Introduction}

At the root of the discipline of analytic combinatorics
is
the fact that if we consider a generating function
analytically as a complex function,
then the asymptotic behaviour of its coefficients can be determined from the singularities on its circle of convergence.
Specifically (see~\cite[Theorem~VI.1]{FS2009}), given a combinatorial class~$\CCC$ 
with ordinary generating function
\[
f_\CCC(z) \eq \sum_{n\geqs0} |\CCC_n| z^n \eq \sum_{\ccc\in\CCC} z^{|\ccc|} ,
\]
where $\CCC_n$ is the set of objects in $\CCC$ of size $n$,
suppose that $f_\CCC(z)$
is (the Taylor series of) a complex function
with unique dominant singularity~$\rho$, such that
\[
f_\CCC(z) \eq g(z) + \frac{h(z)}{(1-z/\rho)^\alpha} ,
\]
where $\alpha\notin-\bbN$, and both $g$ and $h$ are analytic on the closed disk $\overline{D}(0,\rho)$.
Then the 
asymptotic number of objects in $\CCC$ of size $n$ is given by
\begin{equation}\label{eqClassical}
|\CCC_n| \eq
\big[z^n\big]f_\CCC(z) \ssim \frac{h(\rho)}{\Gamma(\alpha)} \+ \rho^{-n} \+ n^{\alpha-1} ,
\end{equation}
where $\big[z^n\big]f(z)$ is the coefficient of $z^n$ in (the Taylor series of) $f(z)$. 

\vspace{-9pt}\subsubsection*{Rational and irrational classes}\vspace{-9pt}

\begin{figure}[t]
\begin{center}
\begin{tikzpicture}[scale=0.9]
\draw[thick,fill=blue!30]   (0,0)      rectangle (1,1);
\draw[thick,fill=blue!30]   (1,0)      rectangle (2,1);
\draw[thick,fill=yellow!50] (2,0)      rectangle (3.414,1);
\draw[thick,fill=blue!30]   (3.414,0)  rectangle (4.414,1);
\draw[thick,fill=yellow!50] (4.414,0)  rectangle (5.828,1);
\draw[thick,fill=yellow!50] (5.828,0)  rectangle (7.243,1);
\draw[thick,fill=blue!30]   (7.243,0)  rectangle (8.243,1);
\draw[thick,fill=blue!30]   (8.243,0)  rectangle (9.243,1);
\draw[thick,fill=blue!30]   (9.243,0)  rectangle (10.243,1);
\draw[thick,fill=blue!30]   (10.243,0) rectangle (11.243,1);
\draw[thick,fill=yellow!50] (11.243,0) rectangle (12.657,1);
\draw (0,-.1)--(12.657,-0.1);
\foreach \x in {0,...,12} {
  \draw (\x,-.1)--(\x,-.225);
  \draw (\x,-.375) node{\scriptsize\x};
}
\end{tikzpicture}
\end{center}
\vspace{-12pt}
\caption{A tiling of a strip using seven tiles of length $1$ and four of length $\sqrt2$}\label{figTilingT}
\end{figure}
In this paper,
we establish analogous results when
we remove the requirement that the size of a combinatorial object must be an integer. 
We require only that the set of sizes of objects in a combinatorial class 
has no accumulation point in~$\bbR$.
That is,
objects in a class $\CCC$ must take their sizes from
some countable set $\Lambda_\CCC$ consisting of a nonnegative increasing sequence of values 
that tends to infinity.
It remains necessary that,
for each $\lambda\in\Lambda_\CCC$,
there are only finitely many objects of size $\lambda$ in~$\CCC$.
We distinguish two cases.

Firstly, if there exists some $\omega>0$ such that $\Lambda_\CCC\subseteq\omega\bbN$ (that is, each object size is a multiple of~$\omega$), then we say that the class $\CCC$ is \emph{rational}.
The sizes of objects in rational classes need not themselves be rational:
every object in $\CCC$ has a rational size if and only if $\omega\in\bbQ$.
Rational classes are not our primary concern. We investigate their asymptotics in Section~\ref{sectRational}.

In contrast, if there is no positive $\omega$ such that $\Lambda_\CCC\subseteq\omega\bbN$, then
we say that $\CCC$ is \emph{irrational}.
As a simple example, consider the class $\TTT=\TTT^{(\beta)}$ of
tilings of a strip of width one
using two types of rectangular tile:
a $1\times1$ square tile, and a tile with dimensions $\beta\times1$ for some $\beta>0$.
The size of a tiling is the length of the tiled strip.
If $\beta$ is irrational, then so is $\TTT$, 
with the set of object sizes being $\{k+\ell\beta:k,\ell\in\bbN\}$.
See Figure~\ref{figTilingT} for an illustration, with $\beta=\sqrt2$.

\vspace{-9pt}\subsubsection*{Ribenboim series}\vspace{-9pt}



To record the sizes of objects in an irrational class,
we use formal power series that admit irrational exponents.
We call these \emph{Ribenboim series} after
Paulo Ribenboim, who, building on the earlier work of
Levi-Civita~\cite{LeviCivita1893}, Hahn~\cite{Hahn1907},
Neumann~\cite{Neumann1949} and Higman~\cite{Higman1952},
investigated their algebraic properties in a series of papers in the 1990s~\cite{ER1990,Ribenboim1991,Ribenboim1992,Ribenboim1994,Ribenboim1995,Ribenboim1997}.
See~\cite{Ehrlich1995} for a historical perspective, and~\cite{KKS2025} for a recent study.

Formally, a Ribenboim series in $z$ with coefficients in a commutative ring $R$ is an element $\sum_{\lambda}c_\lambda z^\lambda$ of the ring $R[[z^{\bbR^{\geq0}}]]$ 
of {Hahn series} with nonnegative exponents such that
the {support} of the exponents $\{\lambda:c_\lambda\neq0\}$ 
contains no finite accumulation point.\footnote{In general, a Hahn series permits exponents whose support is any well-ordered subset of an ordered group.}

For enumeration, we take $R=\bbZ$.
Thus, the 
generating function (\gf{}) of an irrational class~$\CCC$,
which we call an \emph{irrational} generating function (\igf{}),
is the formal Ribenboim series
\[
f_\CCC(z)
\eq \sum_{i\geqs1}
|\CCC_{\lambda_i}|
z^{\lambda_i}
\eq \sum_{\lambda\in\Lambda_\CCC}
|\CCC_\lambda|
z^\lambda
\eq \sum_{\ccc\in\CCC} z^{|\ccc|} ,
\]
where
$\CCC_\lambda\eq\{\ccc\in\CCC:|\ccc|=\lambda\}$ is the set of objects in $\CCC$ of size~$\lambda$,
and
$\lambda_1<\lambda_2<\ldots$
is the total ordering on the elements of~$\Lambda_\CCC$.

For example, the \igf{} for the class $\TTT$ of strip tilings is
\[
f_{\TTT}(z)
\eq \frac1{1\sminus z\sminus z^\beta}
\eq \sum_{k,\ell\geqs0}\tbinom{k+\ell}k z^{k+\ell\beta} .
\]

\vspace{-9pt}\subsubsection*{Asymptotics of irrational classes}\vspace{-9pt}

In general, the coefficients in an \igf{} can fluctuate wildly.
For example, here are seven consecutive terms in the \gf{} for~$\TTT^{(\sqrt2)}$:
\[
8\,z^{1+7\sqrt2},\quad\; z^{11},\quad\; 126\,z^{4+5\sqrt2},\quad\; 120\,z^{7+3\sqrt2},\quad\; z^{8\sqrt2},\quad\; 11\,z^{10+\sqrt2},\quad\; 84\,z^{3+6\sqrt2} .
\]
Thus it makes no sense to consider the asymptotics of the coefficients directly.

Instead, we consider
the number of objects in an irrational class of size at most a given value, which may exhibit
smooth asymptotic behaviour.
Given a class $\CCC$ and any $x\geqs0$, we use $\CCC_{\leqs x}$ to denote the set $\{\ccc\in\CCC:|\ccc|\leqs x\}$ of objects in $\CCC$ of size at most $x$.
Note that the 
structure required of $\Lambda_\CCC$
is sufficient to ensure that each of these sets is finite.
Similarly, if $f(z)\eq \sum_{\lambda\in\Lambda} c_\lambda z^\lambda$ is a Ribenboim series, then
\[
\big[z^{\leqs x}\big]f(z) \eq \sum_{\lambda\leqs x} c_\lambda
\]
extracts the sum of the coefficients of those terms of the series with exponent at most~$x$.
Thus, $\big[z^{\leqs x}\big]f_\CCC(z) \eq |\CCC_{\leqs x}|$.

Analytically, a Ribenboim series
has a radius of convergence in
$[0,\infty]$ and is analytic within its disk of convergence, except on the non-positive real axis, where there is a branch cut.
Moreover, if it has nonnegative coefficients and a finite radius of convergence $\rho$, then it has a singularity at the point $z=\rho$.
See Proposition~\ref{propRibenboimConvergence}.

In order to specify an additional condition that must be satisfied by an irrational class for our main theorem to apply, we introduce a transform that we apply to \igf{}s.

\vspace{-9pt}\subsubsection*{Dirichlet generating functions}\vspace{-9pt}

If $f(z)\eq \sum_{\lambda\in\Lambda} c_\lambda z^\lambda$ is a Ribenboim series,
then we define its \emph{exponential transform} to be the generalised Dirichlet series
\[
F(s) \eq \sum_{\lambda\in\Lambda} c_\lambda e^{-\lambda s} ,
\]
in which $e^{-s}$ has been substituted for~$z$.\footnote{In the special case that $\Lambda=\{\log n:n\geqs1\}$, this is a standard Dirichlet series $\sum_{n\geqs1}a_n\+n^{-s}$; see Section~\ref{sectIntegerFactorizations}.}
The exponential transform of the \igf{} of a combinatorial class is its \emph{Dirichlet generating function} (\dgf{}).

\begin{figure}[t]
\begin{center}
  \begin{tikzpicture}[scale=1.5,>=latex]
    \fill[blue!9] (-1.1,-.75) rectangle (1.1,.75);
    \draw[gray,fill=yellow!75,radius=0.333] (0,0) circle;
    \draw[->] (-1.1,0)--(1.1,0);
    \draw[->] (0,-.75)--(0,.75);
    \draw[ultra thick] (-1.1,0)--(0,0);
    \node at (.4,-.13) {\footnotesize$\rho$};
    \node at (.9,.55) {\footnotesize$z$};
    \fill[radius=.035] (0.333,0) circle;
  \end{tikzpicture}
  \qquad\quad
  \begin{tikzpicture}[scale=2.5,>=latex]
    \draw[white] (0,-.45) rectangle (1,.45);
    \draw[->] (0,0.085)--(1,0.085);
    \node at (.5,.2) {\small$z \,\mapsto\, e^{-s}$};
    \draw[->] (1,-0.085)--(0,-0.085);
    \node at (.5,-.2) {\small$\log(1/z) \,\mapsfrom\, s$};
  \end{tikzpicture}
  \qquad\quad
  \begin{tikzpicture}[scale=0.25,>=latex]
    \fill[blue!4] (-4.6,-4.5) rectangle (8.6,4.5);
    \fill[blue!9] (-4.6,-3.142) rectangle (8.6,3.142);
    \fill[yellow!30] (3,-4.5) rectangle (8.6,4.5);
    \fill[yellow!75] (3,-3.142) rectangle (8.6,3.142);
    \draw[gray] (3,-4.5)--(3,4.5);
    \draw[->] (-4.6,0)--(8.6,0);
    \draw[->] (0,-4.5)--(0,4.5);
    \draw[ultra thick,dotted] (-4.6,3.142)--(8.6,3.142);
    \draw[ultra thick,dotted] (-4.6,-3.142)--(8.6,-3.142);
    \node at (-1.4,-2.4) {\scriptsize$-\pi i$};
    \node at (-1.4,+2.4) {\scriptsize$+\pi i$};
    \node at (3.5,-.9) {\scriptsize$\log(1/\rho)$};
    \node at (7.8,2.2) {\footnotesize$s$};
    \fill[radius=.21] (3,0) circle;
  \end{tikzpicture}
  \vspace{-9pt}
\end{center}
  \caption{The mapping of the region of convergence under the exponential transform}\label{figTransform}
\end{figure}

If a Ribenboim series has radius of convergence $\rho$, then under the exponential transform its cut-disk
region of convergence is mapped into the right open half-plane $\ohplane{\log(1/\rho)}$.
Its circle of convergence is mapped to the \emph{line of convergence} $\re s=\log(1/\rho)$ of the Dirichlet series,
with its dominant singularity $z=\rho$ being mapped to the point $s=\log(1/\rho)$.
See Figure~\ref{figTransform} for an illustration.

If $F(s)$ is the exponential transform of a Ribenboim series $f(z)$, then $F(s)=f(e^{-s})$ as long as $s$~lies in the strip $-\pi<\im s\leqs\pi$.
In general, this identity does not hold for other values of~$s$.
On the other hand, $f(z)=F(\log(1/z))$ for all nonzero $z\in\bbC$.
\cbstart
See Section~\ref{sectTheory}
for how the exponential transform relates to the lifting of 
a Ribenboim series to a Riemann surface.
\cbend

\vspace{-9pt}\subsubsection*{\Pity{}}\vspace{-9pt}

Suppose $\CCC$ is an irrational combinatorial class whose \igf{} has radius of convergence $\rho\in(0,1)$.
Then we say that $\CCC$ is \emph{\prim{}} or has \emph{\pity{}} if
$\log(1/\rho)$ is the unique singularity of its \dgf{} on
its line of convergence. 

With this definition, we can now state our primary result, which enables the extraction of asymptotics from the \igf{} of a \prim{} class, in an analogous manner to~\eqref{eqClassical} above.
\newcommand{\mainTheoremText}
{
Suppose that $\CCC$ is a \primi{} combinatorial class.
  If its irrational generating function $f_\CCC(z)$ has positive radius of convergence~$\rho<1$, and one can write
  \[
  f_\CCC(z) \eq g(z) \splus \frac{h(z)}{(1-z/\rho)^\alpha} ,
  \]
  where $\alpha\notin-\bbN$, 
  and both $g$ and $h$ are analytic on the closed cut disk $\overline{D}(0,\rho) \setminus \bbR^{\leqs0}$
  \cbstart
  and 
  the exponential transform of $g$ is analytic on the closed half-plane $\chplane{\log(1/\rho)}$,
  \cbend
  then
  the asymptotic number of objects in $\CCC$ of size at most $x$ is given by
  \[
  |\CCC_{\leqs x}| \eq
  \big[z^{\leqslant x}\big]f_\CCC(z)
  \;\sim\;
  \frac{h(\rho)}{\log(1/\rho)\,\Gamma(\alpha)} \+ \rho^{-x} \+ x^{\alpha-1} .
  \]
}
\begin{thmO}\label{thmMain}
  \mainTheoremText
\end{thmO}
\cbstart
The requirement that $g$ is not only analytic on the cut disk but also that its exponential transform is analytic on the half-plane
is not overly restrictive in practice, 
since 
functions satisfying the weaker condition but not the stronger one typically only occur in artificially constructed examples.
Indeed, we conjecture that the stronger condition is not actually needed. See Conjecture~\ref{conjAnalytic}.
\cbend

Note that the shape required of an irrational generating function in Theorem~\ref{thmMain} is not sufficient for a class to be primitive.
The unique dominant singularity of an \igf{} may 
correspond to multiple singularities on the line of convergence of the \dgf{}.
In contrast with the smooth asymptotics of \prim{} classes,
irrational classes that are \emph{not} \prim{} may exhibit periodic oscillations in their asymptotics.
An example of such a class is analysed in Section~\ref{sectNotIntrinsicallyIrrational}.


It is interesting to briefly compare Theorem~\ref{thmMain} against the asymptotics of standard combinatorial classes.
If $\AAA$ is a 
class of integer-sized objects with ordinary generating function $f_{\!\AAA}(z)$, then $(1-z)^{-1}f_{\!\AAA}(z)$ is its cumulative generating function:
\[
|\AAA_{\leqs n}| \eq
\big[z^{\leqs n}\big]f_{\!\AAA}(z) \eq
\big[z^n\big]\frac{f_{\!\AAA}(z)}{1-z} .
\]
Thus, if
$f_{\!\AAA}(z)$
has unique dominant singularity~$\rho<1$ such that
\[
f_{\!\AAA}(z) \eq g(z) + \frac{h(z)}{(1-z/\rho)^\alpha} ,
\]
where $\alpha\notin-\bbN$, and $g$ and $h$ are analytic on $\overline{D}(0,\rho)$,
then by~\eqref{eqClassical}, for any nonnegative real~$x$, 
the asymptotic
number of objects in $\AAA$ of size at most $x$ is given by the step function
\[
|\AAA_{\leqs x}| \ssim \frac{h(\rho)}{(1-\rho)\Gamma(\alpha)} \+ \rho^{-\floor{x}} \+ \floor{x}^{\alpha-1} .
\]
If we let $\delta(x)=x-\floor{x}$ denote the fractional part of $x$, then
this matches Theorem~\ref{thmMain} if 
\begin{equation}\label{eqDelta}
  \rho^{\delta(x)} \+ (1-\delta(x)/x)^{\alpha-1} \eq \frac{1-\rho}{\log(1/\rho)} ,
\end{equation}
the solutions of which depend only on the values of $\rho$ and $\alpha$.
If $\rho$ is a simple pole ($\alpha=1$) then this holds exactly whenever the fractional part of $x$ equals
\[
\delta_\rho \defeq \frac
{\log(1/(1-\rho))\:+\:\log\log(1/\rho)}
{\log(1/\rho)} ,
\]
a value that always lies between $0$ and $\frac12$.
When the dominant singularity is not a simple pole, 
since $\delta(x)/x$ tends to zero as $x$ increases,
for each sufficiently large $n\in\bbN$ there exists a $\delta_n\in(0,1)$
for which $x=n+\delta_n$ satisfies~\eqref{eqDelta},
with $\liminfty\delta_n=\delta_\rho$.


\vspace{-9pt}\subsubsection*{Sufficient conditions for \pity{}}\vspace{-9pt}

In order not to have to establish individually the \pity{} of each class in which we are interested,
it is useful to have conditions which guarantee that a class is \prim{}.
The following simple condition (Proposition~\ref{propLinearCond}) is of particular use.

To state it, in an analogous manner to our definition of an irrational class,
we say that a set $S$ of real numbers is \emph{irrational} if there is no $\omega\neq0$ such that $S\subseteq\omega\bbZ$.
In particular, this holds if there exist $x,y\in S$ such that $x/y\notin\bbQ$.

If every singularity $s$ on the line of convergence
of the \dgf{} of a class~$\CCC$
is a root of
an equation of the form
\[
\sum_{\gamma\in\Gamma} c_\gamma \+ e^{-\gamma s} \eq 1,
\]
where $\Gamma$ is an irrational set of positive numbers 
and each $c_\gamma$ is
also
positive,
then $\CCC$ is \prim{}.

\together7
For example, the \dgf{} for the class $\TTT$ of strip tilings,
\[
F_{\TTT}(s)
\eq \frac1{1\sminus e^{-s}\sminus e^{-\beta s}} ,
\]
satisfies this requirement if $\beta\notin\bbQ$.
Hence $\TTT$ is \prim{} if $\beta$ is irrational, and by Theorem~\ref{thmMain} we have 
\[
|\TTT_{\leqs x}| \ssim
\frac{1}{\big(\rho+\beta\+\rho^\beta\big)\log(1/\rho)}
\+ \rho^{-x}
\eq
\frac{\rho^{-x}}{H(\rho)},
\]
where $\rho$ is the unique positive root of
$1-z-z^\beta$,
and
\phantomsection\label{pageEntropy}$H(x)=-x\log x-(1-x)\log(1-x)$,
the latter simplification resulting from the identities
$\rho^\beta=1-\rho$
and $\beta\log\rho=\log(1-\rho)$,
exhibiting a mysterious connection to entropy.

\vspace{-9pt}\subsubsection*{Context}\vspace{-9pt}

The analytic approach to combinatorial enumeration,
expounded and popularised by Flajolet and Sedgewick~\cite{FS2009},
has seen significant recent growth, as evidenced by the publication of several new graduate textbooks taking varied perspectives~\cite{Melczer2021,Mishna2020,PW2013}.

However, we are aware of only one work that considers irrational classes, and that addresses a rather different question:
Garrabrant and Pak~\cite{GP2014} (see also Melczer~\cite[Section~3.4.2]{Melczer2021})
considers the tiling of a strip of length $n+\veps$ with a finite set of irrational length tiles with adjacency constraints,
for some fixed $\veps\in\bbR$ and each $n\in\bbN$.
Their main result is that the ordinary generating functions for these tilings are precisely the diagonals of multivariate $\bbN$-rational functions.
As far as we know, Ribenboim sequences have not previously been used for enumeration, although Richard Stanley
made use of polynomials with irrational exponents in~\cite{Stanley2024}.

The idea of extending the scope of analytic combinatorics to irrational classes was first presented in a MathOverflow question and at a Banff workshop in 2016~\cite{BevanIrrationalMOQuestion,BevanIrrationalBIRS}.
The main result in this paper originally appeared in the Master's thesis of the second author~\cite{Conde2024}.
The authors are particularly grateful to Andrew Elvey Price for discussions and advice, and also to two anonymous referees whose feedback helped us to improve the presentation of this work.
In addition, the authors thank Svante Janson for drawing to our attention errors in the published version of this paper~\cite{BCIrrational} which are corrected in this revision.

Theorem~\ref{thmMain} only gives the first-order asymptotics of \primi{} classes.
With more detailed analysis it ought to be possible to calculate further terms, 
analogous to the result of Flajolet and Odlyzko~\cite{FO1990} that
the coefficient of $z^n$ in $(1-z/\rho)^{-\alpha}$ admits the asymptotic expansion
\begin{equation}\label{eqFlajoletOdlyzko}
\big[z^n\big]\+\frac1{(1-z/\rho)^{\alpha}} \;\sim\; \frac{1}{\Gamma(\alpha)} \+ \rho^{-n} \+ n^{\alpha-1} \+ \bigg(\!1 \++\+ \sum_{k=1}^\infty \frac{e_k}{n^k}\bigg),
\end{equation}
where $e_k$ is a specific polynomial in $\alpha$ of degree~$2k$.
Another direction to explore would be the establishment of central limit theorems (for example, see~\cite{Hwang2001,HJ2011}).
A suitable limit theorem may be of help in determining the asymptotic enumeration of irrational tilings of a square (see Section~\ref{sectFloorTiling}).

\vspace{-9pt}\subsubsection*{Outline}\vspace{-9pt}

The remaining two sections of this paper are independent of each other and can be read in either order.
In Section~\ref{sectApplications},
we present some conditions that guarantee \pity{},
and then illustrate the use of Theorem~\ref{thmMain} with
a wide variety of applications, mostly irrational analogues of classical combinatorial classes.
We also investigate rational 
classes, irrational classes that are not \prim{}, and phase transitions in families of \primi{} classes.
The somewhat shorter Section~\ref{sectTheory} contains the required analytic theory of Ribenboim and Dirichlet series, culminating in a proof of Theorem~\ref{thmMain}.

\section{Applications}\label{sectApplications}


The heart of this section consists of the application of our approach to a variety of \primi{} classes.
We also analyse rational 
classes,
and irrational classes that are not \prim{}.
We begin, in Section~\ref{sectIntrinsicIrrationality}, by establishing conditions under which we can guarantee \pity{},
and we conclude in Section~\ref{sectPhaseTransitions} with an investigation of phase transitions in families of \prim{} classes.

\subsection{Establishing \pity{}}\label{sectIntrinsicIrrationality}

In this section, we present three propositions that can be used to establish that a class is \prim{}.
We begin with a very simple condition ensuring \pity{}, that is of much wider applicability than might be expected.
Its proof is elementary.

\begin{propO}\label{propLinearCond}
  If every singularity $s$ on the line of convergence of the Dirichlet generating function of a combinatorial class~$\CCC$
  is a root of
  an equation of the form
  \begin{equation}\label{eqLinearCond}
  \sum_{\gamma\in\Gamma} c_\gamma \+ e^{-\gamma s} \eq 1 ,
  \end{equation}
  where $\Gamma$ is an irrational set of positive numbers 
  and each $c_\gamma$ is
  also
  positive,
  then $\CCC$ is \prim{}.
\end{propO}
\begin{proof}
  Let $\mu+i\bbR$ be the line of convergence of the \dgf{}.
  Since the \dgf{} has positive coefficients, $\mu$ is a singularity (see Proposition~\ref{propDirichletSingularity}).

  Note that, for any $a,b\in\bbR$, we have $\re e^{a+bi} \sleqs e^a$, with equality only if $b\in2\pi\bbZ$.

  Thus, $\mu+ti$ is a root of \eqref{eqLinearCond}
  only if, for each $\gamma\in\Gamma$, we have
  $t\gamma\in2\pi\bbZ$.
  If $t\neq0$, then this implies that $\frac{t}{2\pi}\Gamma\subseteq\bbZ$, so $\Gamma$ is not irrational.
\end{proof}

%
%

Note that $\Gamma$ does not need to be finite. The sum may be any irrational Dirichlet series with positive coefficients.

Although Proposition~\ref{propLinearCond} is stated in terms of \dgf{}s,
the simplicity of the exponential transform means that
a class can often be seen to be \prim{}
by inspecting its \igf{}.
For example, suppose $\FFF\eq\seq{\GGG}$ consists of sequences of objects from some 
irrational class $\GGG$.
Then the \igf{}s of $\FFF$ and $\GGG$ satisfy the identity
\[
f(z)\eq\frac1{1-g(z)}.
\]
The class $\FFF$ is said to be \emph{supercritical} (see~\cite[Section V.2]{FS2009}) if
$f$ has a dominant singularity at the unique positive root of $g(z)=1$.
In this case, the conditions of Proposition~\ref{propLinearCond} are satisfied by~$\FFF$. 
Thus, we have the following structural condition for \pity{}.

\begin{propO}\label{propSupercritical}
  If $\GGG$ is irrational and $\FFF\eq\seq{\GGG}$ is supercritical, then $\FFF$ is \prim{}.
\end{propO}

When applying Proposition~\ref{propLinearCond}, it is essential to take into account every contribution to the dominant singularity.
For example, a class with \igf{}
\[
\frac1{\big(1-z-3\+z^{\log_2\!6}\big)\big(1-z-5\+z^{\log_2\!10}\big)}
\]
has a double pole dominant singularity at $z=\frac12$ and is \prim{},
both factors being irrational,
whereas a class with \igf{}
\[
\frac1{\big(1-z-3\+z^{\log_2\!6}\big)\big(1-5\+z^{\log_2\!5}\big)} ,
\]
which also has a double pole dominant singularity at $z=\frac12$,
is not \prim{},
the second factor yielding singularities in the \dgf{} at $\log2+2k\pi i/\log_25$, for each $k\in\bbZ$.

Proposition~\ref{propLinearCond} can also be applied when the dominant singularity is a branch point rather than a pole.
See Section~\ref{sectQuadratic} for several examples.

In general, it is not possible to weaken
the requirement that each $c_\gamma$ in \eqref{eqLinearCond} must be positive.
For example, a class with \igf{}
\[
\frac1{1 - 2\+z - z^\beta + 2\+z^{1+\beta}} \eq \frac1{\big(1-2z\big)\big(1-z^\beta\big)},
\]
which has its dominant singularity at $z=\frac12$ (for any positive $\beta$),
is not \prim{},
its \dgf{} having singularities at $\log2+2k\pi i$ ($k\in\bbZ$).

However, when the dominant singularity satisfies an equation of the form $\sum_{\gamma\in\Gamma} c_\gamma \+ z^\gamma \eq 1$ in which
some of the coefficients are negative, it is sometimes possible to rearrange the \igf{} so that it then meets the requirements of Proposition~\ref{propLinearCond}.
See Section~\ref{sectAdjacencyConstraints} for one approach of wide applicability.

Here is another condition guaranteeing \pity{}, that can be deduced from Proposition~\ref{propLinearCond} by rearrangement:

\begin{propO}\label{propNonlinearCond}
Suppose $\mu+i\bbR$ is the line of convergence of the Dirichlet generating function of a combinatorial class~$\CCC$.
  Suppose that every singularity on this line is a root of an equation of the form
  \[
  \bigg(1 - \sum_{\delta\in\Delta} a_\delta \+ e^{-\delta s}\bigg)^{\!\!k} \sminus \sum_{\gamma\in\Gamma} c_\gamma \+ e^{-\gamma s} \eq 0 ,
  \]
  for some integer $k\geqs2$, where
  each $\gamma\in\Gamma$ and each $\delta\in\Delta$ is positive,
  as is each $c_\gamma$ and each $a_\delta$.
  If the set $\Gamma\cup\Delta$ is irrational, then $\CCC$ is \prim{}.
\end{propO}
\begin{proof}
  Let $G(s)=\sum_{\gamma\in\Gamma} c_\gamma \+ e^{-\gamma s}$
  and $H(s)=\sum_{\delta\in\Delta} a_\delta \+ e^{-\delta s}$.

  Every singularity $s$ on the line of convergence of the \dgf{} satisfies the equation
  \[
  1
  \eq \frac{G(s)}{\big(1-H(s)\big)^k}
  \eq G(s)\+\big(1+H(s)+H(s)^2+\ldots\big)^k
  \eq \sum_{\lambda\in\Lambda} b_\lambda\+ e^{-\lambda s} ,
  \]
  for certain $b_\lambda$, all positive, where
  \[
  \Lambda \eq \big\{ \gamma+\beta : \gamma\in\Gamma,\, \beta\in\mathrm{span}(\Delta) \big\} ,
  \]
  where we use $\mathrm{span}(\Delta)$ to denote the set of
  finite $\bbN$-linear combinations of elements of $\Delta$.

  The set $\Lambda$ is irrational, since $\Gamma\cup\Delta$ is.
  Thus the conditions of Proposition~\ref{propLinearCond} are satisfied and $\CCC$ is \prim{}.
\end{proof}

\subsection{Tiling and packing}\label{sectTiling}


In this section, we illustrate the scope of our approach by exploring the asymptotic enumeration of a variety of types of tiling and packing,
including the establishment of a general result concerning the \pity{} of tilings with adjacency constraints.
We also look at
the structure of a typical tiling,
at irrational tilings which are not \prim{},
and at how an irrational class of tilings relates to a classical number-theoretic problem concerning integer factorizations.
We conclude with a natural open question that seems to require some additional ideas.

\vspace{-9pt}\subsubsection{Strips of width 1}\vspace{-9pt}

Generalising the example from the introduction,
given
an
irrational set of tile lengths $\Gamma\subset\bbR^{>0}$,
let $\TTT^\Gamma$ be the class of tilings of a strip of width one
with tiles of dimensions $\{\gamma\times1:\gamma\in\Gamma\}$, the size of a tiling being the length of the tiled strip.
This class has \igf{}
\[
f_{\TTT^\Gamma}(z) \eq
\frac1{1\sminus\sum_{\gamma\in\Gamma} z^{\gamma}} ,
\]
and is easily seen to be \prim{} by Proposition~\ref{propLinearCond}.
Applying Theorem~\ref{thmMain} then yields
\begin{equation}\label{eqTiling}
\big|\TTT^\Gamma_{\leqs x}\big| \ssim \frac{1}{\log(1/\rho)\+\sum_{\gamma\in\Gamma} \gamma \+\rho^{\gamma}} \+ \rho^{-x},
\end{equation}
where $\rho$ is the unique positive root of
$\sum_{\gamma\in\Gamma} z^{\gamma}=1$.


\begin{figure}[t]
\begin{center}
\begin{tikzpicture}[scale=0.9]
\fill[white]            (-0.787,0)  rectangle (0,1);
\draw[thick,fill=blue!30]   (0,0)      rectangle (1,1);
\draw[thick,fill=yellow!50] (1,0)      rectangle (2.414,1);
\draw[thick,fill=blue!30]   (2.414,0)  rectangle (3.414,1);
\draw[thick,fill=yellow!50] (3.414,0)  rectangle (4.828,1);
\draw[thick,fill=yellow!50] (4.828,0)  rectangle (6.243,1);
\draw[thick,fill=blue!30]   (6.243,0)  rectangle (7.243,1);
\draw[thick,fill=blue!30]   (7.243,0)  rectangle (8.243,1);
\draw[thick,fill=blue!30]   (8.243,0)  rectangle (9.243,1);
\fill[yellow!25]              (9.243,0)  rectangle (10.657,1);
\draw[thick,fill=yellow!50]   (9.243,0)  rectangle (9.8696,1);  
\end{tikzpicture}
\end{center}
\vspace{-12pt}
\caption{A tiling with a final partial tile of a strip of length $\pi^2$, using square tiles and tiles of length~$\sqrt2$}\label{figTilingPartialTile}
\end{figure}

\vspace{-9pt}\subsubsection{Tilings with a final partial tile}\vspace{-9pt}

The set 
$\TTT^\Gamma_{\leqs x}$
consists of \emph{partial} tilings of a strip of length $x$, where each tiling $\ttt$ has a gap of length $x-|\ttt|$ (which may be zero) at the end.
It seems somewhat more natural to consider tilings in which the final tile (only) may be a \emph{partial tile}, a tile which may have any length less than the greatest tile length in~$\Gamma$.
See Figure~\ref{figTilingPartialTile} for an illustration.
(We consider a two-dimensional variant below in Section~\ref{sectFloorTiling}.)

Note that a tiling with a possible final partial tile may have \emph{any} nonnegative real size.
Thus sets of such tilings do not themselves have Ribenboim series \igf{}s.

Suppose $\Gamma=\{\gamma_1<\gamma_2<\ldots<\gamma_k\}$ is finite.
There are two possibilities to consider.
Either tiles can only be distinguished by their length, or else they are distinguishable in some other way (for example, by colour).

Let $\UUU^\Gamma$ be the set of tilings with a possible final partial tile using tiles indistinguishable except by their length.
Then, since a partial tile can have any length up to $\gamma_k$, for any $x\geqs\gamma_k$ we have
\[
\big|\UUU^\Gamma_x\big| \eq \big|\TTT^\Gamma_{\leqs x}\big| \sminus \big|\TTT^\Gamma_{\leqs x-\gamma_k}\big| ,
\qquad \text{and thus} \qquad
\big|\UUU^\Gamma_x\big| \ssim \big(1-\rho^{\gamma_k}\big)\big|\TTT^\Gamma_{\leqs x}\big| .
\]
In contrast, let $\VVV^\Gamma$ be the set of tilings with a possible final partial tile using tiles distinguishable by colour.
The number of possible colours for a final partial tile depends on its length, since a partial tile of colour $i$ must have length less than~$\gamma_i$.
Thus for any $x\geqs\gamma_k$ we have
\[
\big|\VVV^\Gamma_x\big| \eq
\sum_{i=1}^{k}(k+1-i)\left(
\big|\TTT^\Gamma_{\leqs x-\gamma_{i-1}}\big|-\big|\TTT^\Gamma_{\leqs x-\gamma_i}\big|
\right)
\eq
k\big|\TTT^\Gamma_{\leqs x}\big| \sminus \sum_{i=1}^k \big|\TTT^\Gamma_{\leqs x-\gamma_i}\big| ,
\]
where $\gamma_0=0$. 
Here, the factor $k+1-i$ is the number distinct tiles which are longer than the space available for the partial tile, which has length at least $\gamma_{i-1}$ but less than~$\gamma_i$.
Only tiles that are too long can be cut down to form a partial tile.

Since $\sum_{\gamma\in\Gamma} \rho^{\gamma}=1$, after simplification, we have
\[
\big|\VVV^\Gamma_x\big| \ssim \big(|\Gamma|-1\big) \big|\TTT^\Gamma_{\leqs x}\big|,
\]
a result which begs for a more direct combinatorial explanation.
Indeed, we have the following identity:

\begin{propO}
The number of coloured tilings with a possible final partial tile of a strip of length $x$ satisfies
\[
\big|\VVV^\Gamma_x\big| \eq \big(|\Gamma|-1\big) \big|\TTT^\Gamma_{<x}\big| \splus 1 .
\]
\end{propO}
\begin{proof}
  Each tiling in $\VVV^\Gamma_x$ can be created uniquely from a strict partial tiling $\ttt \in \TTT^\Gamma_{<x}$ by extending it with final nonempty run of tiles of the same colour,
  where this colour differs from the colour of the last tile in $\ttt$.
  There are $|\Gamma|-1$ choices for this colour, unless $\ttt$ is empty in which case any of the colours may be chosen.
  Note that we need $|\ttt|<x$ since we need space for the final monochrome run of tiles.
\end{proof}


\begin{figure}[t]
\begin{center}
\begin{tikzpicture}[scale=0.9]
\draw[thick,fill=yellow!50] (-4.243,0)      rectangle (-2.828,1);
\draw[thick,fill=yellow!50] (-2.828,0)      rectangle (-1.414,1);
\draw[thick,fill=yellow!50] (-1.414,0)      rectangle (0,1);
\draw[thick,fill=blue!30]   (0,0)      rectangle (1,1);
\draw[thick,fill=yellow!50] (1,0)      rectangle (2.414,1);
\draw[thick,fill=blue!30]   (2.414,0)  rectangle (3.414,1);
\draw[thick,fill=yellow!50] (3.414,0)  rectangle (4.828,1);
\draw[thick,fill=white] (4.828,0)  rectangle (5.6266,1);
\end{tikzpicture}
\end{center}
\vspace{-12pt}
\caption{A maximal packing of a strip of length $\pi^2$, using square tiles and tiles of length~$\sqrt2$}\label{figPacking}
\end{figure}

\vspace{-9pt}\subsubsection{Maximal packings}\vspace{-9pt}

Complementary to these tilings, are \emph{maximal packings} of a strip, $\PPP^\Gamma$.
A maximal packing of a strip of length $x$ consists of a partial tiling of the strip having length greater than $x-\gamma_1$ so it can't be extended to a longer partial tiling.
See Figure~\ref{figPacking} for an illustration.
Hence,
\[
\big|\PPP^\Gamma_x\big| \eq \big|\TTT^\Gamma_{\leqs x}\big| \sminus \big|\TTT^\Gamma_{\leqs x-\gamma_1}\big| ,
\qquad \text{and thus} \qquad
\big|\PPP^\Gamma_x\big| \ssim \big(1-\rho^{\gamma_1}\big)\big|\TTT^\Gamma_{\leqs x}\big| .
\]

\vspace{-9pt}\subsubsection{The structure of a typical tiling}\vspace{-9pt}

The distribution of the value of a parameter can be extracted from a bivariate \igf{} in the normal way (see~\cite[Proposition~III.2]{FS2009}).
Specifically, if $f(z,u)$ is the bivariate \igf{} of a class $\CCC$ in which $u$ marks the value of some parameter $\upsilon:\CCC\to\bbR$ (which could take any real values), then the mean value of $\upsilon$ over objects of size at most $x$ is given by
\begin{equation}\label{eqExpec}
\expecn{\leqs x}{\upsilon}
\eq
\frac{\big[z^{\leqs x}\big]f_u(z,1)}{\big[z^{\leqs x}\big]f(z,1)} .
\end{equation}
Higher moments and concentration results can be established by taking higher derivatives.

Given a tiling $\ttt\in\TTT^\Gamma$ and some $\beta\in\Gamma$, let $L_\beta(\ttt)$ be the sum of the lengths of the $\beta$-length tiles in $\ttt$. 
Then, the 
bivariate
\igf{} in which $u$ marks~$L_\beta$ 
is
\[
f
(z,u) \eq
\frac1{1-u^\beta\+z^\beta-\sum_{\gamma\in\Gamma\setminus\{\beta\}} z^{\gamma}} ,
\]
and
\[
f_u(z,1) \eq
\frac{\beta\+z^\beta}{\big(1\sminus\sum_{\gamma\in\Gamma} z^{\gamma}\big)^2}
.
\]
By Proposition~\ref{propLinearCond}, 
this is
\prim{}.
Using~\eqref{eqExpec}, and applying Theorem~\ref{thmMain} then yields the asymptotic expectation for~$L_\beta$: 
\[
\expecn{\leqs x}{L_\beta}
\ssim
\frac{\beta\+\rho^\beta}{\sum_{\gamma\in\Gamma} \gamma \+\rho^{\gamma}}\+x
.
\]
Therefore,
the asymptotic ratio of the contribution of $\alpha$-length and $\beta$-length tiles to a tiling equals ${\alpha\+\rho^\alpha}/{\beta\+\rho^\beta}$.

If $f(z,u)$ is the bivariate \igf{} of a \prim{} class, in which $u$ marks the value of a parameter $\upsilon$, then
it is not immediately obvious that $\expecn{\leqs x}{\upsilon}$
gives information about the structure of \emph{large} objects in the class.
However,
if the class grows exponentially, and $\expecn{\leqs x}{\upsilon}$ is polynomial in $x$, then it is easy to check that for any $c>0$ we have
\[
\expecn{\leqs x}{\upsilon}
\ssim
\frac{\big[z^{\leqs x}\big]f_u(z,1) - \big[z^{\leqs x-c}\big]f_u(z,1)}{\big[z^{\leqs x}\big]f(z,1)-\big[z^{\leqs x-c}\big]f(z,1)} ,
\]
the asymptotics of $\expecn{\leqs x}{\upsilon}$ thus being the same as
the mean value of $\upsilon$ for objects whose size lies in the half-open interval $(x-c,x]$.



\vspace{-9pt}\subsubsection{Ordered integer factorizations}\vspace{-9pt}\label{sectIntegerFactorizations}

The asymptotics given in equation~\eqref{eqTiling} above do not require the set of tiles to be finite.
Suppose $\Gamma=\{\log k:k\geqs2\}$.
Then $\big|\TTT^\Gamma_{\log n}\big|$, as well as being
the number of tilings
of a strip of length $\log n$
using tiles of logarithmic length, is also the number of \emph{ordered factorizations} of~$n$.
For example, the integer 12 has eight distinct ordered factorizations:
\[
12 \qquad 6\!\times\!2 \qquad 4\!\times\!3 \qquad 3\!\times\!4 \qquad 2\!\times\!6 \qquad 3\!\times\!2\!\times\!2 \qquad 2\!\times\!3\!\times\!2 \qquad 2\!\times\!2\!\times\!3 .
\]
This is a classical problem,
the question of establishing the asymptotics of $|\TTT^\Gamma_{\leqs\log x}\big|$ being first investigated by Kalm\'ar in the 1930s
(see also entries \oeis{A074206} and \oeis{A173382} in the OEIS~\cite{OEIS}).
The \igf{} and \dgf{} for these tilings can be expressed nicely in terms of the Riemann zeta function:
\[
F_{\TTT^\Gamma}(s)
\eq
\frac1{1-\sum_{k\geqs2}e^{-s\log k}}
\eq
\frac1{2-\zeta(s)}
.
\]
Our method immediately yields the result first deduced by Kalm\'ar~\cite{Kalmar1931,Kalmar1931a} that the asymptotic number of ordered factorizations of integers less than or equal to~$x$ is
\[
|\TTT^\Gamma_{\leqs\log x}\big|
\ssim
-\frac1{\mu\+\zeta'(\mu)}\+x^\mu \sapprox 3.14294\+ x^{1.72865} ,
\]
where $\mu$ is the positive real root of the equation $\zeta(s)=2$.
And,
using the trivariate \igf{}
\[
f_{\TTT^\Gamma}(x,u,v)
\eq
\frac1{2 - (u-1)z^{\log m} - v\+\zeta(\log(1/z))} ,
\]
we can also establish
that the expected number of factors in an ordered factorization of an integer at most $x$ is asymptotically
\[
-\frac2{\zeta'(\mu )} \log x \sapprox 1.10002 \log x ,
\]
and that the expected number of factors equal to a fixed integer $m$ is
\[
-\frac{m^{-\mu }}{\zeta'(\mu )}\log x \sapprox 0.55001 \+m^{-1.72865}\log x .
\]
For example, the proportion of factors equal to 2 is asymptotically $2^{-\mu}/2\approx15.086\%$.

For a much more detailed analysis of ordered factorizations, see Hwang~\cite{Hwang2000}.

\vspace{-9pt}\subsubsection{Infinite sets of rational tiles}\vspace{-9pt}

We do not need irrational tile lengths for a class of tilings to be \prim{}, since
an infinite set of \emph{rational} tile lengths will be irrational if their denominators are not bounded.

As an example, suppose $\Gamma=\{k+2^{-k}:k\geqs0\}=\{1,\,1\frac12,\,2\frac14,\,3\frac18,\,\ldots\}$.
The \igf{} of $\TTT^\Gamma$ is not a Puiseux series, as would be the case for a rational class (see Section~\ref{sectRational}), because denominators are unbounded.
Solving numerically, equation~\eqref{eqTiling} yields the approximate asymptotics
for $\big|\TTT^\Gamma_{\leqs x}\big|$ of $0.69622 \, x^{2.31329}$.

\vspace{-9pt}\subsubsection{Adjacency constraints}\label{sectAdjacencyConstraints}\vspace{-9pt}

Suppose we have three types of tile, of length $\alpha$, $\beta$ and $\gamma$, with $\{\alpha,\beta,\gamma\}$ irrational.
If a \mbox{$\beta$-length} tile is not permitted to immediately follow an \mbox{$\alpha$-length} tile, then the \igf{} for the class of permitted strip tilings is
\[
f(z)
\eq
\frac1{1-z^\alpha-z^\beta-z^\gamma+z^{\alpha+\beta}} .
\]
The denominator of this \igf{} does not have the right structure to use either of Propositions~\ref{propLinearCond} or~\ref{propNonlinearCond} directly to establish \pity{}.
However, a consideration of the combinatorial structure enables a rearrangement of the \igf{} that does meet the requirements of Proposition~\ref{propLinearCond}:
\[
f(z)
\eq
\frac1{ \left(1-z^\beta-\frac{z^\gamma}{1-z^\alpha}\right)(1-z^\alpha)}
\eq
\frac1{ \left(1-z^\beta-\sum_{k\geqs0}z^{k\alpha+\gamma}\right)(1-z^\alpha)} .
\]
Here, we use the fact that the set of valid tilings can be represented by the regular expression $(\beta+\alpha^{\!\STAR}\gamma)^{\!\STAR}\alpha^{\!\STAR}$.
The dominant singularity comes from the first factor in the denominator, so by Proposition~\ref{propLinearCond} we know that these tilings are \prim{} and can extract the asymptotics using Theorem~\ref{thmMain}.

This approach is of very general applicability.
Any non-degenerate class of strip tilings with a finite number of adjacency constraints can be defined by a regular expression (or equivalently, by a deterministic finite automaton), in which
the dominant singularity comes from sequences of sub-tilings from some (possibly infinite) set $\SSS$ (see \cite[Sections~V.3 and~V.5]{FS2009}).
If $\SSS$ is irrational, then the class of tilings is \prim{}.
Note that to check that $\SSS$ is irrational only requires finding a single pair of sub-tilings $\sss_1,\sss_2\in\SSS$ such that $|\sss_1|/|\sss_2|\notin\bbQ$.


\vspace{-9pt}\subsubsection{Irrational tilings that are not \prim{}}\label{sectNotIntrinsicallyIrrational}\vspace{-9pt}

\begin{figure}[t]
\begin{center}
\begin{tikzpicture}[scale=0.9]
\draw[thick,fill=blue!10]   (0,0)  rectangle (1,1);
\draw[thick,fill=blue!30]   (1,0)  rectangle (2,1);
\draw[thick,fill=blue!30]   (2,0)  rectangle (3,1);
\draw[thick,fill=blue!10]   (3,0)  rectangle (4,1);
\draw[thick,fill=blue!30]   (4,0)  rectangle (5,1);
\draw[thick,fill=blue!10]   (5,0)  rectangle (6,1);
\draw[thick,fill=yellow!75] (6,0)      rectangle (7.414,1);
\draw[thick,fill=yellow!75] (7.414,0)  rectangle (8.828,1);
\draw[thick,fill=yellow!25] (8.828,0)  rectangle (10.243,1);
\draw[thick,fill=yellow!75] (10.243,0) rectangle (11.657,1);
\end{tikzpicture}
\end{center}
\vspace{-12pt}
\caption{A tiling in $\WWW^{(\sqrt2)}$}\label{figTilingW}
\end{figure}

Let $\WWW=\WWW^{(\beta)}$ consist of tilings of a strip of width one with four types of rectangular tile:
two different colours of $1\times1$ square tiles, and
two different colours of $\beta\times1$ tiles, with the restriction that all the square tiles must occur to the left of the tiles of length $\beta$.
Suppose $\beta>1$ and that $\beta$ is irrational.
See Figure~\ref{figTilingW} for an example.
These tilings have the \igf{},
\[
f_\WWW(z) \eq
\frac1{(1-2z)(1-2z^\beta)} ,
\]
with radius of convergence~$\frac12$.
However, $\WWW$ is not \prim{}, because its \dgf{},
\[
F_\WWW(s) \eq
\frac1{(1-2e^{-s})(1-2e^{-\beta s})} ,
\]
has evenly spaced poles at $\log2+2\pi i\+\bbZ$. 

This class is simple enough to analyse explicitly.
Indeed, we have
\[
|\WWW_{\leqs x}| \eq
\sum_{\ell=0}^{\floor{x/\beta}} 2^\ell \sum_{k=0}^{\floor{x-\ell\beta}} 2^k,
\]
where $\ell$ is the number of tiles of length $\beta$, and $k$ is the number of square tiles.

\begin{figure}[t]
  \centering
  \includegraphics[width=2.7in]{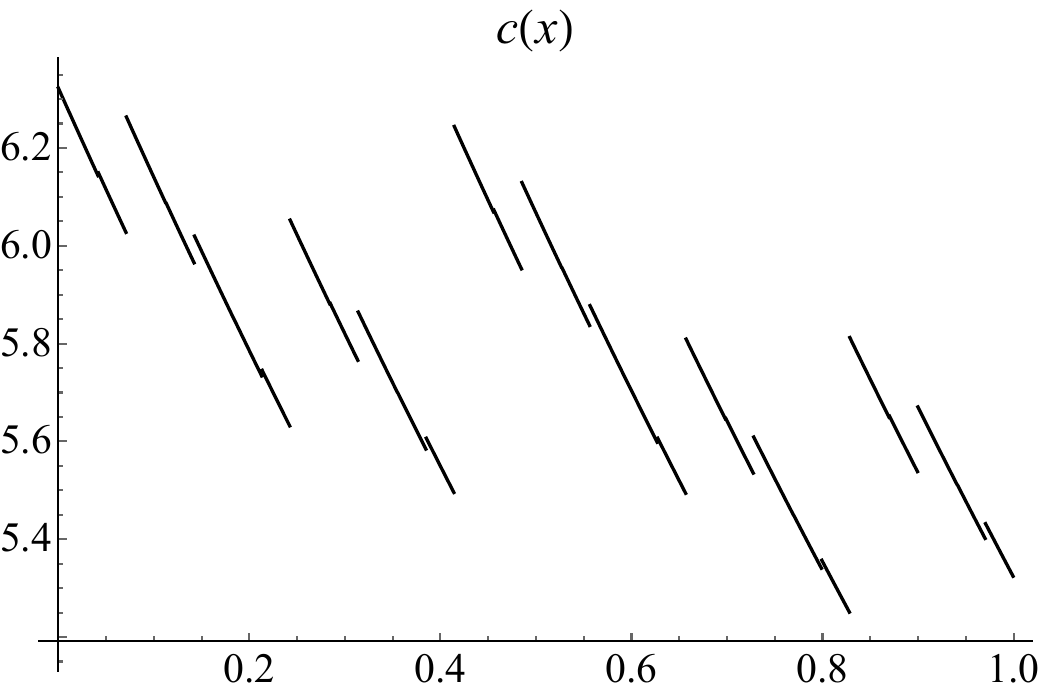}
  \caption{$\Big|\WWW^{(\sqrt2)}_{\leqs x}\Big| \ssim c(x)\+2^x$}\label{figOscillation}
\end{figure}

Suppose, for each $\ell\geqs 0$, we let $\delta_\ell(x)=(x-\ell\beta)\bmod 1$.
That is, $\delta_\ell(x)$ is the fractional part of $x-\ell\beta$.
Then,
after some simplification,
it can be seen that asymptotically,
\[
|\WWW_{\leqs x}| \ssim
c(x)2^x,
\qquad
\text{where~~}
c(x)\eq
\sum_{\ell=0}^{\infty} 2^{1-(\beta-1)\ell-\delta_\ell(x)}
\]
is a bounded function that oscillates with period one.
See Figure~\ref{figOscillation} for the plot of one period of $c(x)$ when $\beta=\sqrt2$.
Note that $c(x)$ is, in fact,
not continuous on any interval, since 
each
term in the sum contributes a discontinuity at $\ell\beta\bmod 1$, 
and these are dense on the unit interval since $\beta$ is irrational.

In light of this behaviour, and also that of \emph{rational} classes (see Section~\ref{sectRational}),
we make the following 
conjecture concerning those classes whose \dgf{} has evenly spaced dominant singularities.
\begin{conjO}\label{conj}
  Suppose that $\CCC$ is a combinatorial class
  whose generating function $f_\CCC(z)$ has radius of convergence~$\rho<1$, and
  \[
  f_\CCC(z) \eq g(z) \splus \frac{h(z)}{(1-z/\rho)^\alpha} ,
  \]
  where $\alpha\notin-\bbN$, and both $g$ and $h$ are analytic on the closed cut disk $\overline{D}(0,\rho) \setminus \bbR^{\leqs0}$ 
  \cbstart
  and 
  the exponential transform of $g$ is analytic on the closed half-plane $\chplane{\log(1/\rho)}$.
  \cbend

  If the only singularities on the line of convergence of its Dirichlet generating function, $F_\CCC(s)$, are evenly spaced at
  $\log(1/\rho)+2\pi i\+\bbZ/\omega$,
  for some constant $\omega$,
  then
  the asymptotic number of objects in $\CCC$ of size at most $x$ is given by
  \[
  |\CCC_{\leqs x}|
  \;\sim\;
  c(x) \+ \rho^{-x} \+ x^{\alpha-1} ,
  \]
  where $c(x)$ is a bounded function that oscillates with period~$\omega$.
\end{conjO}


However, it \emph{is} possible for a class of irrational tilings that are not primitive to have smooth aperiodic asymptotics; that is, for $c(x)$ to be asymptotically constant.
For example, suppose $\QQQ$ consists of strip tilings with two distinct tiles of length $\log 2$ and three distinct tiles of length $\log 3$, in which all the $\log 2$ tiles must be to the left of any $\log 3$ tiles.
Then $\QQQ$ has \igf{}
\[
f_\QQQ(z) \eq
\frac1{(1-2\+z^{\log2})(1-3\+z^{\log3})} ,
\]
which has radius of convergence~$e^{-1}$.
This class of tilings is not \prim{} since its \dgf{},
\[
F_\QQQ(s) \eq
\frac1{(1-2\+e^{-s\log2})(1-3\+e^{-s\log3})} ,
\]
has a pole of order two at $1+2\pi i$ for each $i\in\bbZ$.

Now, 
\[
f_\QQQ(z) \eq \sum_{a\geqs0}\+ \sum_{b\geqs0} 2^a\+ 3^b\+ z^{a\log 2 + b\log 3} 
\eq \sum_{n\in S_3} n\+ z^{\log n} ,
\] 
where $S_3$ is the set of 3-smooth numbers, 
a positive integer being \emph{$k$-smooth} if none of its prime factors exceeds~$k$, 
so the 3-smooth numbers are those of the form $2^a3^b$ for~$a,b\geqs0$.

Thus, $|\QQQ_{\leqs x}|$ gives the sum of the 3-smooth numbers no greater than~$e^x$.
Hence (see~\cite{Tenenbaum2015} for example), despite $\QQQ$ not being primitive, we have
\[
|\QQQ_{\leqs x}| \ssim
\frac{x\+e^x}{(\log2)(\log3)} ,
\]
which
matches the asymptotics that would be given by Theorem~\ref{thmMain} if $\QQQ$ were in fact primitive.


\vspace{-9pt}\subsubsection{Floor tiling}\label{sectFloorTiling}\vspace{-9pt}

Consider the process of tiling a rectangular floor using rectangular tiles of dimension $\beta\times\gamma$, where \mbox{$\beta/\gamma\notin\bbQ$}.
Tiles may be laid in either orientation, and
tiling starts in the southwest corner, partial tiles being permitted (only) along the east and north edges of the floor.
See Figure~\ref{figSquareTiling} for an example using tiles whose aspect ratio is the golden ratio~$\vphi$.

\begin{figure}[t]
  \centering
  \includegraphics[width=1.75in]{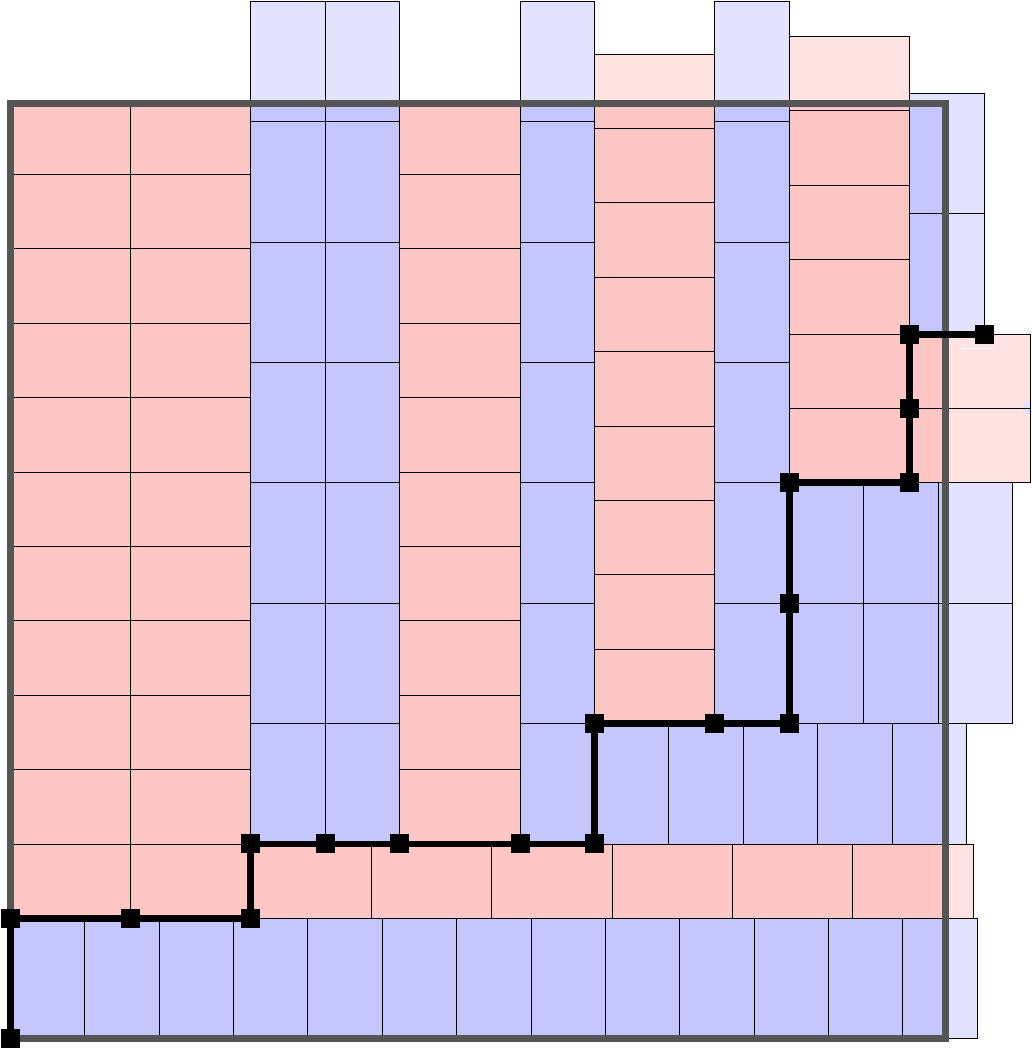}
  \caption{A tiling of a square with sides of length $4\pi$ by rectangular tiles of dimension $1\times\vphi$}\label{figSquareTiling}
\end{figure}

\newcommand{\uub}{\mathbf{u}_\beta}
\newcommand{\uug}{\mathbf{u}_\gamma}
\newcommand{\vvb}{\mathbf{v}_\beta}
\newcommand{\vvg}{\mathbf{v}_\gamma}

Because the aspect ratio of the tiles is irrational,
when tiling from the southwest,
in order to be able to complete the tiling,
at each step a ``ray'' of identically oriented tiles needs to be laid either northwards or eastwards.
Each ray can be identified by a line segment across its end, and these form a walk with step set
$
S \eq
\{
\uub=(\beta,0),\, \uug=(\gamma,0),\, \vvb=(0,\beta),\, \vvg=(0,\gamma)
\}
$.
A tiling is completed once the walk reaches either the east or the north edge of the rectangular floor.
See Figure~\ref{figSquareTiling} for an illustration.

Note, however, that the sub-walks $\uub\vvg$ and $\vvg\uub$ correspond to the same union of two rays, and this is also the case for $\uug\vvb$ and $\vvb\uug$.
Thus tilings are in bijection with walks with step set $S$ that avoid consecutive pairs of steps $\vvg\uub$ and $\vvb\uug$ (that is, right turns with steps of different lengths).
For example, in Figure~\ref{figSquareTiling} the two steps forming each right turn have the same length.

Thus, the bivariate \igf{} for these walks satisfies the functional equation
\[
f(u,v) \eq
1\:+\:f(u,v) \Big((1-v^\gamma)u^\beta \:+\: (1-v^\beta)u^\gamma\:+\: v^\beta\:+\: v^\gamma \Big),
\]
where $u$ marks the distance east, and $v$ marks the distance north, from the southwest corner. 
Hence we have
\[
f(u,v) \eq \frac1{1 - u^\beta - u^\gamma - v^\beta - v^\gamma + u^\beta v^\gamma + u^\gamma v^\beta} .
\]
To determine the asymptotic number of tilings of a square floor, we need asymptotics for $\big[u^{\leqs x}\big]\big[v^{\leqs x}\big]f(u,v)$, and probably also for $\big[u^{\leqs x}\big]\big[v^{\leqs y}\big]f(u,v)$, at least when $y$ is (very) close to~$x$.
This seems to require techniques considerably
beyond the scope of this work, so we leave it as an open problem.
Handling rectangular floors with aspect ratio not equal to one is likely to require a generalisation to the irrational case of the approach to ``off-diagonal'' asymptotics established in
\emph{Analytic Combinatorics in Several Variables}~\cite{PW2013} (see also~\cite{Melczer2021,Mishna2020}).
\begin{questO}
  (a)~Asymptotically, how many tilings are there of an $x\times x$ square using rectangular tiles of dimension $\beta\times\gamma$, when $\beta/\gamma\notin\bbQ$?
  (b)~Asymptotically, how many such tilings are there of an $x\times rx$ rectangle, for any fixed $r$?
\end{questO}
For~(a), on the basis of numerical data and anticipated concentration behaviour, we would expect the answer to be $C_{\beta,\gamma}\+\rho^{-2x}/\sqrt{x}$, for some constant $C_{\beta,\gamma}$ to be determined, where $\rho$ is the radius of convergence of~$f(z,z)$.

\subsection{Lattice walks and trees}\label{sectQuadratic}


In this section, we explore some structures whose \igf{}s have dominant square root singularities.
We begin with irrational Motzkin paths, and Dyck paths with an infinite step set,
before analysing two classes of plane trees with irrational vertex sizes.

\begin{figure}[ht]
  \centering
  \begin{tikzpicture}[scale=0.4]
    \setplotptradius{0.125}
    \draw[thin,gray!50] (1,1)--(19,1);
    \plotpermnobox{}{1,1,2,1,2,3,2,3,3,4,3,2,2,2,1,1,2,2,1}
    \draw[thick] (1,1)--(2,1)--(3,2)--(4,1)--(6,3)--(7,2)--(8,3)--(9,3)--(10,4)--(12,2)--(14,2)--(15,1)--(16,1)--(17,2)--(18,2)--(19,1);
  \end{tikzpicture}
  \caption{A Motzkin path in $\MMM^{(\sqrt2)}$ of Euclidean length $6+12\sqrt2$}\label{figMotzkin}
\end{figure}

\vspace{-9pt}\subsubsection{Motzkin paths}\vspace{-9pt}

Given some irrational $\beta>0$, let $\MMM=\MMM^{(\beta)}$ be the class of Motzkin paths in which up-steps and down-steps have size $\beta$, and horizontal steps have size $1$.
Graphically, the size of a step could be its width, or its height, or it could be its Euclidean length.
For example, if $\beta=\sqrt2$, then the size of a path in $\MMM$ equals its Euclidean length when drawn in the standard way; see Figure~\ref{figMotzkin} for an example.
By the first passage decomposition, the \igf{} for $\MMM$ satisfies the functional equation
\[
f_\MMM(z)\eq 1 \splus z\+ f_\MMM(z) \splus z^{2\beta}\+ f_\MMM(z)^2.
\]
This yields
\[
f_\MMM(z)\eq \frac{1-z-\sqrt{(1-z)^2-4\+z^{2\beta}}}{2\+z^{2\beta}} ,
\]
which has a dominant square root singularity.
Given that the discriminant is the difference of two squares, this singularity is
at the unique positive root $\rho$ of $1-z-2\+z^\beta$ (the other factor being positive for $0\leqs z\leqs1$).
Thus, every singularity on the line of convergence of the \dgf{} is a root of $1-e^{-s}-2\+e^{-\beta s}$.
So, by Proposition~\ref{propLinearCond}, $\MMM$ is \prim{}.

Now,
\[
f_\MMM(z) \eq \frac{1-z}{2\+z^{2\beta}} \:-\: \frac{h(z)}{(1-z/\rho)^{-1/2}}, \quad \text{where~~} h(z) \eq 
\frac1{2\+z^{2\beta}}\+{\sqrt{\dfrac{(1-z)^2-4\+z^{2\beta}}{1-z/\rho}}} ,
\]
and so has the required form. 
Applying Theorem~\ref{thmMain} then gives
\begin{align*}
|\MMM_{\leqs x}|
&\ssim
-\frac1{2\+\rho^{2\beta}\+\log(1/\rho)\+\Gamma(-\tfrac12)} \+\sqrt{\lim_{z\to\rho^-}\frac{(1-z)^2-4\+z^{2\beta}}{1-z/\rho}} \+\rho^{-x} \+x^{-3/2} \\[6pt]
& \eq
\frac{\sqrt2}{\log(1/\rho)} \+\sqrt{\frac{\rho + \beta\+(1-\rho)}{\pi\+(1-\rho)^3}} \+\rho^{-x} \+x^{-3/2}
,
\end{align*}
where we use the fact that $2\rho^\beta=1-\rho$.

For $\beta=\sqrt2$, we have the approximate asymptotics $\big|\MMM^{(\sqrt2)}_{\leqs x}\big|\simeq2.29313\cdot2.39330^xx^{-3/2}$.



%
%

\vspace{-9pt}\subsubsection{Dyck paths with infinite step sets}\vspace{-9pt}

Enumeration by Euclidean length enables us to handle infinite step sets.
In a (normal) Dyck path, each up-step is associated, in either the first passage decomposition or the arch decomposition, with a matching down-step, being the next down-step to the right at the same level (see~\cite[pages~76--77]{FS2009}, for example). 
The classes we consider in this section require that a matched pair of up and down steps have the same height.

As our first example, a \mbox{\emph{tall-Dyck}} path consists of a sequence of steps drawn from the set $\{(1,\pm k) : k\geqs1\}$, which ends at the level of its start point, never going below this level, and in which each $(1,k)$ up-step is {matched} by a $(1,-k)$ down-step.
That is, a tall-Dyck path is formed from a classical Dyck path by replacing each matched pair of up and down steps by a $(1,k)$ step and a $(1,-k)$ step for some~$k$.
See Figure~\ref{figTallDyck} for an example.

\begin{figure}[t]
  \centering
  \begin{tikzpicture}[scale=0.4]
    \setplotptradius{0.125}
    \draw[thin,gray!75] (1,1)--(19,1);
    \plotpermnobox{}{1,3,4,3,6,3,1,5,1,3,1,2,4,5,4,2,3,2,1}
    \draw[thick] (1,1)--(2,3)--(3,4)--(4,3)--(5,6)--(6,3)--(7,1)--(8,5)--(9,1)--(10,3)--(11,1)--(12,2)--(13,4)--(14,5)--(15,4)--(16,2)--(17,3)--(19,1);
  \end{tikzpicture}
  \caption{A tall-Dyck path}
  \label{figTallDyck}
\end{figure}

Let $\DDD$ be the class of tall-Dyck paths, and let $f_\DDD(z,w)$ be its bivariate \igf{} enumerating by Euclidean length in which $w$ marks width.
Note that there are infinitely many tall-Dyck paths of any given even width.
Then, by the first passage decomposition, $f^\DDD(z,w)$ satisfies
\[
f^\DDD(z,w)
\eq
1 \splus
w^2 \sum_{k\geqs1} z^{2\sqrt{1+k^2}} \+f^\DDD(z,w)^2 .
\]
Hence,
\[
f^\DDD(z,w)
\eq
\frac{1-\sqrt{1-4\+ w^2 s(z)}}{2\+ w^2 s(z)} ,
\]
where $s(z)=\sum_{k\geqs1} z^{2\sqrt{1+k^2}}$.

The univariate \igf{} $f^\DDD(z,1)$ has a dominant square root singularity,
$\rho\approx0.529999$,
at the unique positive root of $1-4\+s(z)$, so by Proposition~\ref{propLinearCond}, $\DDD$ is \prim{}.
Applying Theorem~\ref{thmMain} then gives
\[
|\DDD_{\leqs x}|
\ssim
\frac{2\+\sqrt{\rho\+s'(\rho)}}{\sqrt{\pi}\log(1/\rho)}\+\rho^{-x} \+x^{-3/2}
\sapprox
1.69800\cdot1.88680^xx^{-3/2} .
\]
The asymptotic expected width of a tall-Dyck path of Euclidean length at most $x$ is
\[
\frac{\big[z^{\leqs x}\big]{f^\DDD}_{\!\!w}(z,1)}{\big[z^{\leqs x}\big]f^\DDD(z,1)}
\ssim
\frac{x}{2\+\rho\+s'(\rho)}
\sapprox
0.547797 x
.
\]

\begin{figure}[ht]
  \centering
  \begin{tikzpicture}[scale=0.4]
    \setplotptradius{0.125}
    \draw[thin,gray!75] (1,1)--(26,1);
    \foreach \x in {1,...,26} {
      \draw[thin,gray!75] (\x,1)--(\x,1.25);
    }
    \draw[thin,gray!75] (1,1)--(1,6);
    \foreach \x in {1,...,6} {
      \draw[thin,gray!75] (1,\x)--(1.25,\x);
    }
    \plotpermnobox{}{1,3,0,0,0,1,2,0,0,5,2,0,3,0,5,3,6,0,3,0,0,2,4,2,0,1}
    \draw[thick] (1,1)--(2,3)--(6,1)--(7,2)--(10,5)--(11,2)--(13,3)--(15,5)--(16,3)--(17,6)--(19,3)--(22,2)--(23,4)--(24,2)--(26,1);
  \end{tikzpicture}
  \caption{A Dyck path with step set $\bbZ^{>0}\times(\bbZ\!\setminus\!\{0\})$}\label{figDyck2}
\end{figure}

This approach
is valid for variants of
Dyck paths with any step set $S\subset\bbR^{>0}\times(\bbR\!\setminus\!\{0\})$ with an irrational set of lengths,
in which each $(h, k)$ up-step in a path is matched by some $(j, -k)$ down-step of the same height.
See Figure~\ref{figDyck2} for an example.
The analysis follows \emph{mutatis mutandis},
by setting
\[
s(z)
\eq
\sum_{k\geqs1} \bigg(\sum_{(h,k)\in S}z^{\sqrt{h^2+k^2}}\bigg) \bigg(\sum_{(j,-k)\in S}z^{\sqrt{j^2+k^2}}\bigg)
.
\]
For example, the asymptotic number of Dyck paths
with step set $\bbZ^{>0}\times(\bbZ\!\setminus\!\{0\})$
of Euclidean length at most $x$ is approximately $1.22713\cdot2.67151^xx^{-3/2}$.

The enumeration of lattice paths in which matched up and down steps do not need to have the same height seems to be rather more complicated.



\vspace{-9pt}\subsubsection{Trees with leaves of irrational size}\vspace{-9pt}

Given some irrational $\beta>0$, let $\AAA=\AAA^{(\beta)}$ be the class of plane trees in which non-leaf vertices have size $1$ and leaves have size $\beta$.
See the left of Figure~\ref{figTrees} for an example.
The \igf{} for $\AAA$ satisfies the functional equation
\[
f_{\!\AAA}(z) \eq
z^\beta + \frac{z\+ f_{\!\AAA}(z)}{1-f_{\!\AAA}(z)} ,
\]
since a tree in $\AAA$ is either a single leaf of size~$\beta$ or else a nonempty sequence of subtrees (each from~$\AAA$) rooted by a vertex of size~$1$.
So,
\[
f_{\!\AAA}(z) \eq
\frac12\! \left(1-z+z^\beta-\sqrt{\left(1-z+z^\beta\right)^{\!2} - 4\+z^\beta}\right) \!.
\]
The discriminant factorises as
\[
\big(1-z^{\beta /2}-\sqrt{z}\big) \big(1-z^{\beta /2}+\sqrt{z}\big) \big(1+z^{\beta /2}-\sqrt{z}\big) \big(1+z^{\beta /2}+\sqrt{z}\big) \!,
\]
with the first factor being smaller than the other factors and decreasing from $1$ for $z\geqs0$.
Thus, $f_{\!\AAA}(z)$ has a dominant square root singularity $\rho$ at the unique positive root of $1-z^{\beta /2}-\sqrt{z}$.
So, by Proposition~\ref{propLinearCond}, $\AAA$ is \prim{}, and Theorem~\ref{thmMain} gives
\[
|\AAA_{\leqs x}|
\ssim
\frac{\sqrt{\rho + \sqrt{\rho} \big(\beta  (1-\sqrt\rho)^2-\rho \big)}}{2 \sqrt{\pi} \log(1/\rho)}
\+\rho^{-x}
\+x^{-3/2}
.
\]




\begin{figure}[t]
  \centering
  \begin{tikzpicture}[scale=0.6]
    \draw[thick,gray!70!black] (2,4)--(2,3)--(3,2)--(4,1)--(6,2)--(5,3)--(5,4);
    \draw[thick,gray!70!black] (1,2)--(4,1)--(4,2);
    \draw[thick,gray!70!black] (6,3)--(6,2)--(8,3)--(9,4);
    \draw[thick,gray!70!black] (3,2)--(3,3);
    \draw[thick,gray!70!black] (8,3)--(7,4);
    \draw[thick,gray!70!black] (8,3)--(8,4);
    \setplotptradius{0.125}
    \plotpermnobox{}{0,3,2,1,3,2,0,3}
    \setplotptradius{0.15}
    \plotpermnobox[blue!50!black]{}{2,4,3,2,4,3,4,4,4}
  \end{tikzpicture}
  \qquad\qquad\qquad
  \begin{tikzpicture}[scale=0.6]
    \draw[thick,gray!70!black] (1,4)--(2,3)--(2,2)--(6,1)--(9,2)--(11,3)--(12,4);
    \draw[thick,gray!70!black] (4,4)--(5,3)--(6,2)--(7,3)--(7,4);
    \draw[thick,gray!70!black] (8,3)--(9,2)--(9,4);
    \draw[thick,gray!70!black] (2,3)--(2,4);
    \draw[thick,gray!70!black] (2,2)--(3,3);
    \draw[thick,gray!70!black] (5,3)--(5,4);
    \draw[thick,gray!70!black] (6,1)--(6,3);
    \draw[thick,gray!70!black] (9,2)--(10,3);
    \draw[thick,gray!70!black] (11,3)--(11,4);
    \setplotptradius{0.1489}
    \plotpermnobox[red]{}{0,2,3,4,4,2,3,0,3,0,4,4}
    \setplotptradius{0.1326}
    \plotpermnobox[blue!50!black]{}{0,3,0,0,3,1,4,3,2,0,3}
    \setplotptradius{0.1197}
    \plotpermnobox[green!50!black]{}{4,4,0,0,0,3,0,0,4,3}
  \end{tikzpicture}
  \caption{At the left, a plane tree in class $\AAA^{(\beta)}$ with leaves of size $\beta$, and, at the right, a plane tree in class~$\BBB^{\alpha,\beta,\gamma}$ with three sizes of leaf and two sizes of non-leaf vertex}
  \label{figTrees}
\end{figure}

\vspace{-9pt}\subsubsection{Trees with three sizes of leaf}\vspace{-9pt}

Given $\alpha,\beta,\gamma>0$, let $\BBB=\BBB^{\alpha,\beta,\gamma}$ be the class of plane trees
in which
each leaf has size $\alpha$, $\beta$ or~$\gamma$,
and
each non-leaf vertex has size $\alpha$ or~$\beta$.
See the right of Figure~\ref{figTrees} for an example.
The \igf{} for $\BBB$ satisfies the functional equation
\[
f_\BBB(z) \eq
z^\alpha + z^\beta + z^\gamma + \frac{(z^\alpha + z^\beta)\+ f_\BBB(z)}{1-f_\BBB(z)} .
\]
So,
\[
f_\BBB(z) \eq
\frac12\! \left(1+z^\gamma-\sqrt{\left(1-z^\gamma\right)^2 - 4\+z^\alpha  - 4\+ z^\beta}\right) \!
,
\]
with dominant square root singularity $\rho$ at the least positive root of $\left(1-z^\gamma\right)^2 - 4\+z^\alpha  - 4\+ z^\beta$.
Proposition~\ref{propLinearCond} does not apply, but if the set $\{\alpha, \beta, \gamma\}$ is irrational
then $\BBB$ is \prim{} by Proposition~\ref{propNonlinearCond}, and we can use Theorem~\ref{thmMain} to determine the asymptotics.

For example, if $\alpha=\log16\approx2.77$, $\beta=\log9\approx2.20$ and $\gamma=\log6\approx1.79$, then $\rho=e^{-1}$ and we have
\[
|\BBB_{\leqs x}|
\ssim
\frac1{12}
\sqrt{\frac{23 \log2+21 \log3}{2 \pi }}
\+e^x
\+x^{-3/2}
.
\]

\subsection{Rational and shifted-rational classes}\label{sectRational}


Recall from the Introduction that we say that a combinatorial class is rational if there exists some $\omega>0$ (not necessarily rational) such that the size of each object in the class is a multiple of~$\omega$.
Its \emph{period} is the largest such~$\omega$.
That is, $\omega$ is the greatest common divisor of the object sizes, being the largest value that divides into each size an integral number of times.
For example, the class of strip tilings with tiles of length $\frac\pi{15}$, $\frac\pi6$ and $\frac\pi4$ is rational with period~$\frac\pi{60}$.

We say that a class is \emph{shifted-rational} if there exist $\omega,\delta>0$ with $\delta/\omega\notin\bbQ$ such that the size of each object in the class equals $n\omega+\delta$ for some $n\in\bbN$.
Given a shifted-rational class with set of object sizes~$\Lambda$,
its {period} is the largest $\omega$ such that
there exists a positive \emph{offset} $\delta<\omega$ such that $\Lambda\subseteq\omega\bbN+\delta$.
A rational class is considered to have offset zero.
Note that we consider a shifted-rational class to be an irrational class. 
Among other things, this simplifies the statement of Proposition~\ref{propSupercritical}.
Note that
if $\GGG$ is a shifted-rational class and $\FFF\eq\seq{\GGG}$ is supercritical, then $\FFF$ is \prim{}.

Suppose $f(z)$ is the \igf{} of a rational or shifted-rational class with period $\omega$ and offset $\delta$. Then $f(z)=z^\delta j\big(z^\omega\big)$, where $j(z)$ is an ordinary power series.
Suppose that $f(z)$ has positive radius of convergence~$\rho<1$, and one can write
  \[
  f(z) \eq g(z) \splus \frac{h(z)}{(1-z/\rho)^\alpha} ,
  \]
  where $\alpha\notin-\bbN$, and both $g$ and $h$ are analytic on the closed cut disk $\overline{D}(0,\rho) \setminus \bbR^{\leqs0}$.
Then $j(z)$ has a unique dominant singularity at $\rho^\omega$, and
\[
j(z) \eq
z^{-\delta/\omega}
\left(
g(z^{1/\omega})+\frac{h(z^{1/\omega})}{(1-z^{1/\omega}/\rho)^\alpha}
\right)
\!.
\]
Using the classical result~\eqref{eqClassical}, and the fact that
\[
\lim_{z\to\rho^\omega} \, \frac{1-z/\rho^\omega}{1-z^{1/\omega}/\rho} \;=\; \omega ,
\]
we can extract the asymptotics for the coefficients of $f(z)$:
\[
\big[z^{n\omega+\delta}\big]f(z)
\eq
\big[z^n\big]j(z)
\ssim
\frac{\rho^{-\delta}\+ \omega^\alpha\+ h(\rho)}{\Gamma(\alpha)}\+\rho^{-n\omega}\+n^{\alpha-1}
\ssim
\frac{\omega\+ h(\rho)}{\Gamma(\alpha)}\+\rho^{-n\omega-\delta}\+(n\omega+\delta)^{\alpha-1} .
\]
For an ordinary power series $j(z)$ and integer $n$, the cumulative sum of the coefficients satisfies the identity
\[
\big[z^{\leqs n}\big]j(z)
\eq
\big[z^n\big]\frac{j(z)}{1-z} .
\]
Thus, the asymptotic number of objects of size at most $x$ in a rational or shifted-rational class with \igf{} $f(z)$ is given by 
\[
\big[z^{\leqslant x}\big]f(z)
\ssim
\frac{\rho^{-\delta}\+ \omega\+ h(\rho)}{(1-\rho^\omega)\+\Gamma(\alpha)}\+\rho^{-\floor{(x-\delta)/\omega}\omega-\delta}\+x^{\alpha-1} .
\]
Further terms in the asymptotics of shifted-rational classes can be established by using the result of Flajolet and Odlyzko (see equation~\eqref{eqFlajoletOdlyzko} on page~\pageref{eqFlajoletOdlyzko}).

Note that, if we let the period tend to zero (conceptually tending towards being an irrational class), then
we have $\lim_{\omega\to0}(1-\rho^{\omega})/\omega = -\log\rho$, in line with what we would expect from Theorem~\ref{thmMain}.




Plane trees with edges of irrational size provide a simple example of a shifted-rational class.
Given irrational $\gamma>0$, let $\EEE=\EEE^{(\gamma)}$ be the class of plane trees in which each vertex has size~$1$ and each edge has size~$\gamma$.
Since the number of edges in a tree is one less than the number of vertices,
the set of tree sizes in $\EEE$ is $\{1+(1+\gamma)n:n\in\bbN\}$.
The \igf{} for $\EEE$
is
\[
f_{\EEE}(z) \eq
\frac1{2\+z^\gamma}\! \left(1-\sqrt{1-4\+z^{1+\gamma}}\right) 
\eq
z + z^{2+\gamma}+2 z^{3+2\gamma}+5 z^{4+3\gamma}+14 z^{5+4\gamma} 
+\ldots
\]
Thus,
\[
\big|\EEE_{1+(1+\gamma)n}\big| \eq C_n \eq \frac1{n+1}\binom{2n}{n} \ssim \frac1{\sqrt{\pi}}\+4^n n^{-3/2}
,
\]
where $C_n$ is the $n$th Catalan number.

Suppose $\FFF$ consists of forests of trees with irrational edges, so $\FFF\eq\seq{\EEE}$.
Then the \igf{} for $\FFF$
is
\[
f_{\FFF}(z) \eq
\frac1{1-f_\EEE(z)} \eq
\frac{1-2\+ z^\gamma+\sqrt{1-4\+ z^{1+\gamma}}}{2 (1-z-z^\gamma)} .
\]
For all irrational $\gamma$, the simple pole dominates the square root singularity, so the class is supercritical, and hence \prim{}.


\subsection{Phase transitions}\label{sectPhaseTransitions}

Consider a family $\{\CCC^\bb:\bb\in \Omega\}$ of combinatorial classes, in which each class depends on a vector $\bb\in\bbR^d$ of parameter values, drawn from the set $\Omega\subset\bbR^d$ that consists of those vectors $\bb$ for which $\CCC^\bb$ is \prim{}.
Suppose that for each $\bb\in \Omega$, the asymptotic number of objects in $\CCC^\bb$ is given by
\[
\big|\CCC^\bb_{\leqs x}{}\big|
\ssim
c_\bb \cdot \rho_\bb^{-x} \cdot x^{\nu_\bb} .
\]
Typically, we find that the exponential growth rate $\rho_\bb^{-1}$ varies continuously for parameter values in $\Omega$:
\[
\lim_{\bg\underset{\Omega}{\to}\bb} \rho_\bg \eq \rho_\bb .
\]
Moreover,
$\Omega$ can be partitioned into several regions of continuity $\Omega=\biguplus_{i=1}^k R_i$ corresponding to different phases, such that within each region
the degree $\nu_\bb$ of the polynomial factor (also known as the \emph{critical exponent}) remains constant,
while
the constant $c_\bb$ term varies continuously.
That is, for each $i$ and any $\bb,\bg\in R_i$, we have
$\nu_\bb=\nu_\bg$, and
\[
\lim_{\bg\underset{R_i}{\to}\bb} c_\bg \eq c_\bb .
\]
Phase transitions occur in the asymptotics on the boundaries of the regions, where the value of $\nu_\bb$ changes abruptly.
We illustrate this phenomenon with three examples.

\vspace{-9pt}\subsubsection{Strip tilings}\vspace{-9pt}

\begin{figure}[t]
\begin{center}
\begin{tikzpicture}[scale=0.9]
\draw[thick,fill=white]   (0.000,0)   rectangle (1.000,1);
\draw[thick,fill=blue!30] (1.000,0)   rectangle (1.408,1);
\draw[thick,fill=blue!30] (1.408,0)   rectangle (1.816,1);
\draw[thick,fill=blue!30] (1.816,0)   rectangle (2.224,1);
\draw[thick,fill=white]   (2.224,0)   rectangle (3.224,1);
\draw[thick,fill=blue!30] (3.224,0)   rectangle (3.632,1);
\draw[thick,fill=white]   (3.632,0)   rectangle (4.632,1);
\draw[thick,fill=white]   (4.632,0)   rectangle (5.632,1);
\draw[thick,fill=blue!30] (5.632,0)   rectangle (6.040,1);
\draw[thick,fill=white]   (6.040,0)   rectangle (7.040,1);
\draw[thick,fill=green!90!black]  (7.040,0)   rectangle (7.747,1);
\draw[thick,fill=green!90!black]  (7.747,0)   rectangle (8.454,1);
\draw[thick,fill=white]   (8.454,0)   rectangle (9.454,1);
\draw[thick,fill=green!90!black]  (9.454,0)   rectangle (10.161,1);
\end{tikzpicture}
\end{center}
\vspace{-12pt}
\caption{A tiling in $\YYY^{(\sqrt{1/6},\sqrt{1/2})}$}\label{figTilingY}
\end{figure}

Given $\beta,\gamma>0$, with $\min(\beta,\gamma)$ irrational,
let $\YYY=\YYY^{(\beta,\gamma)}$ consist of tilings of a strip with white tiles of length~$1$,
blue tiles of length~$\beta$, and
green tiles of length~$\gamma$, subject to the requirement that every green tile occurs to the right of any blue tiles.
See Figure~\ref{figTilingY} for an illustration.
The \igf{} for $\YYY$ is
\[
f_\YYY(z) \eq \frac1{1-z-z^\beta} \left( 1+\frac{z^\gamma}{1-z-z^\gamma} \right) \eq \frac{1-z}{(1-z-z^\beta)(1-z-z^\gamma)} ,
\]
with poles, $\rho_\beta$ and $\rho_\gamma$, the positive roots of $1-z-z^\beta$ and $1-z-z^\gamma$, respectively,
with $\rho_\beta<\rho_\gamma$ if $\beta<\gamma$.
Since $\min(\beta,\gamma)$ is irrational, $\YYY$ is \prim{}.

This family of tilings has three phases, with the asymptotics exhibiting a discontinuity across the ray $\beta=\gamma$,
on which there is a linear subexponential term as a result of the double pole.
Either side of this ray, the subexponential term is constant.
\[
\big|\YYY_{\leqs x}\big|
\ssim
\begin{cases}
\dfrac{1-\rho_\beta}
{(1-\rho_\beta-\rho_\beta^\gamma)\+H(\rho_\beta)}\+\rho_\beta^{-x},
& \text{if $\beta<\gamma$}, \\[17pt]
\dfrac{(1-\rho)\+\log(1/\rho)}{H(\rho)^2} \+ \rho^{-x} \+ x,
& \text{if $\beta=\gamma$, where $\rho=\rho_\beta=\rho_\gamma$}, \\[14pt]
\dfrac{1-\rho_\gamma}
{(1-\rho_\gamma-\rho_\gamma^\beta)\+H(\rho_\gamma)}\+\rho_\gamma^{-x},
& \text{if $\beta>\gamma$},
\end{cases}
\]
where $H(x)=-x\log x-(1-x)\log(1-x)$. 


\vspace{-9pt}\subsubsection{Forests of rooted binary trees}\vspace{-9pt}

\begin{figure}[t]
  \centering
  \begin{tikzpicture}[scale=0.5]
    \draw[thin,gray!75] (0,1)--(16,1);
    \draw[thick,gray!70!black] (2,1)--(2,2)--(1,3);
    \draw[thick,gray!70!black] (2,2)--(3,3);
    \draw[thick,gray!70!black] (4,1)--(4,2);
    \draw[thick,gray!70!black] (7,1)--(7,2)--(6,3)--(5,4);
    \draw[thick,gray!70!black] (7,2)--(8,3);
    \draw[thick,gray!70!black] (6,3)--(7,4);
    \draw[thick,gray!70!black] (10,1)--(10,2)--(11,3)--(10,4)--(9,5);
    \draw[thick,gray!70!black] (10,2)--(9,3);
    \draw[thick,gray!70!black] (11,3)--(12,4);
    \draw[thick,gray!70!black] (10,4)--(11,5);
    \draw[thick,gray!70!black] (14,1)--(14,2)--(13,3);
    \draw[thick,gray!70!black] (14,2)--(15,3);
    \setplotptradius{0.125}
    \plotpermnobox{}{3,2,3,1,0,3,2,3,5,4,3,4,0,2}
    \plotpermnobox{}{0,0,0,0,0,0,0,0,0,2,0,0,0,1}
    \setplotptradius{0.14}
    \plotpermnobox[green!50!black]{}{0,1,0,0,0,0,1,0,0,1}
    \setplotptradius{0.15}
    \plotpermnobox[blue!50!black]{}{0,0,0,2,4,0,4,0,3,0,5,0,3,0,3}
  \end{tikzpicture}
  \caption{A forest of rooted binary trees in $\FFF^{(\beta,\gamma)}$}
  \label{figForest}
\end{figure}

Given $\beta,\gamma>0$, consider the following class of plane trees.
Each tree has a \emph{root} of degree one, of size either $1$ or~$\gamma$.
Each other vertex of degree one is a \emph{leaf} and has size either $1$ or~$\beta$.
Every tree has a root and at least one leaf.
Any further (internal) vertices all have degree three and size~$1$.
Let $\FFF=\FFF^{(\beta,\gamma)}$ be the class of forests (that is, sequences) of these trees.
See Figure~\ref{figForest} for an example.

The \igf{} for these forests is
\[
f_\FFF(z)
\eq
\frac{2\+z}{z-z^\gamma+(z+z^\gamma) \sqrt{1-4\+z(z+z^\beta)}} .
\]
The behaviour of the asymptotics and the structure of a typical forest depend primarily on the value of $\gamma$.
The class $\FFF^{(\beta,\gamma)}$~is \prim{} if either $\gamma<1$ and $\gamma$ is irrational or else $\gamma\geqs1$ and $\beta$ is irrational.
Let $\rho_\beta$ be the positive root of $1-4z^2-4z^{1+\beta}$, and, if $\gamma<1$ (so $z^\gamma>z$ for small $z$), let $\rho_\gamma$ be the positive root of
$z^\gamma -z-(z+z^\gamma)\sqrt{1-4\+z(z+z^\beta)}$.
Then, the asymptotics of this family of forests exhibit the following three phases, with three distinct critical exponents: 
\[
\big|\FFF_{\leqs x}\big|
\ssim
\begin{cases}
c_1 \+ \rho_\gamma^{-x} ,
 & \text{if $\gamma < 1$}, \\[3pt]
c_2 \+ \rho_\beta^{-x} \+ x^{-1/2} ,
 & \text{if $\gamma = 1$}, \\[3pt]
c_3 \+ \rho_\beta^{-x} \+ x^{-3/2} ,
 & \text{if $\gamma > 1$},
\end{cases}
\]
for constants $c_1$, $c_2$ and $c_3$, which depend only on $\beta$ and $\gamma$ (for $c_1$ and $c_3$).

The structure of a typical forest also depends on the phase.
If $\gamma<1$, then, asymptotically, a forest contains lots of small trees whose average size is finite.
In contrast, if $\gamma>1$, then a typical forest contains a finite number of large trees.
On the critical line, in a forest of size at most $x$, the number of trees and the average size of a tree both grow like~$\sqrt{x}$.
Formally, if $T(\fff)$ is the number of trees in a forest~$\fff$, then we have
\[
\expecn{\leqs x}{T}
\ssim
\begin{cases}
t_1 x
, & \text{if $\gamma < 1$}, \\[3pt]
t_2\sqrt{x}
, & \text{if $\gamma = 1$}, \\[3pt]
t_3
, & \text{if $\gamma > 1$},
\end{cases}
\]
where $t_1$, $t_2$ and $t_3$ depend only on $\beta$ and~$\gamma$.


\vspace{-9pt}\subsubsection{Grounded Dyck paths}\vspace{-9pt}

We conclude by taking a brief look at a simple example that has slightly more complex critical behaviour.
Given $\beta,\gamma>0$, a \emph{grounded Dyck path} consists of a nonempty Dyck path in which each step has width either $1$ or~$\gamma$.
This path is both preceded and followed by a (possibly empty) sequence of horizontal steps, each of width either $1$ or~$\beta$.
The size of a grounded Dyck path is its horizontal extent, the sum of the widths of its steps.
Such a path can be considered to represent a range of mountains rising from a plain.
See Figure~\ref{figPhasePlane} for two examples.

The class $\GGG=\GGG^{(\beta,\gamma)}$ of grounded Dyck paths has \igf{}
\[
f_\GGG(z)
\eq
\frac{1-2(z+z^\gamma)^2-\sqrt{1-4(z+z^\gamma)^2}}{2(z+z^\gamma)^2\+(1-z-z^\beta)^2} .
\]
Let $\rho_\beta$ be the positive root of $1-z-z^\beta$, and $\rho_\gamma$ be the positive root of $1-2z-2z^\gamma$.
The class $\GGG^{(\beta,\gamma)}$ is \prim{} if 
$\rho_\beta\leqs\rho_\gamma$ and $\beta$ is irrational or $\rho_\gamma\leqs\rho_\beta$ and $\gamma$ is irrational.
The asymptotics of this family of paths exhibit a discontinuity across the critical curve given by the parametric equation
\[
(\beta,\gamma) \eq \big( \log_\rho\!\+(1-\rho) ,\, \log_\rho\!\+(\tfrac12-\rho) \big) ,
\]
for $0<\rho<\frac12$.

\begin{figure}[t]
\centering
\hspace{12pt}
\raisebox{72pt}{$\rho_\beta\leqs\rho_\gamma$}
\hspace{-60pt}
\raisebox{96pt}{\begin{tikzpicture}[scale=0.2]
\setplotptradius{0.175}
\plotpt{2.6309}{0}
\plotpt{4.2618}{0}
\plotpt{5.2618}{0}
\plotpt{6.2618}{0}
\plotpt{7.8927}{0}
\plotpt{8.8927}{0}
\plotpt{10.5236}{0}
\draw[thin,gray!75] (11.5236,0)--(16.1903,0);
\plotpt{11.5236}{0}
\plotpt{13.1903}{1}
\plotpt{14.1903}{0}
\plotpt{15.1903}{1}
\plotpt{16.1903}{0}
\plotpt{17.8212}{0}
\plotpt{19.4521}{0}
\plotpt{20.4521}{0}
\plotpt{22.083}{0}
\plotpt{23.083}{0}
\plotpt{24.083}{0}
\draw[thick] 
(2.6309,0)--(4.2618,0)--(5.2618,0)--(6.2618,0)--(7.8927,0)--(8.8927,0)--(10.5236,0)--(11.5236,0)--(13.1903,1)--(14.1903,0)--(15.1903,1)--(16.1903,0)--(17.8212,0)--(19.4521,0)--(20.4521,0)--(22.083,0)--(23.083,0)--(24.083,0);
\end{tikzpicture}}
\hspace{-144pt}
\includegraphics[width=2.7in]{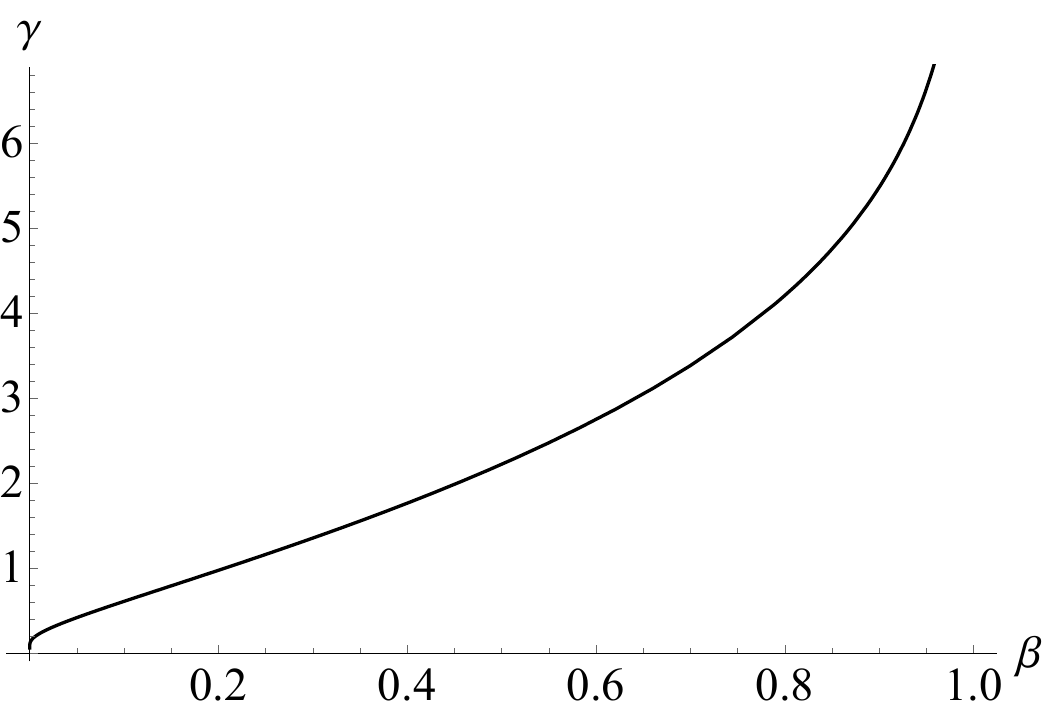}
\hspace{-126pt}
\raisebox{24pt}{\begin{tikzpicture}[scale=0.2]
\setplotptradius{0.175}
\draw[thin,gray!75] (1.,0)--(19.2983,0);
\plotpt{1.}{0}
\plotpt{1.75}{0}
\plotpt{2.75}{0}
\plotpt{3.75}{1}
\plotpt{4.75}{0}
\plotpt{5.1191}{1}
\plotpt{6.1191}{2}
\plotpt{6.4882}{1}
\plotpt{6.8573}{2}
\plotpt{7.2264}{3}
\plotpt{8.2264}{2}
\plotpt{9.2264}{1}
\plotpt{9.5955}{0}
\plotpt{10.5955}{1}
\plotpt{10.9646}{0}
\plotpt{11.3337}{1}
\plotpt{12.3337}{2}
\plotpt{13.3337}{3}
\plotpt{13.7028}{2}
\plotpt{14.0719}{1}
\plotpt{14.441}{2}
\plotpt{14.8101}{1}
\plotpt{15.8101}{0}
\plotpt{16.1792}{1}
\plotpt{17.1792}{2}
\plotpt{18.1792}{1}
\plotpt{18.5483}{0}
\plotpt{19.2983}{0}
\draw[thick] (1.,0)--(1.75,0)--(2.75,0)
--(3.75,1)--(4.75,0)
--(5.1191,1)--(6.1191,2)--(6.4882,1)--(6.8573,2)--(7.2264,3)--(8.2264,2)--(9.2264,1)--(9.5955,0)
--(10.5955,1)--(10.9646,0)
--(11.3337,1)--(12.3337,2)--(13.3337,3)--(13.7028,2)--(14.0719,1)--(14.441,2)--(14.8101,1)--(15.8101,0)
--(16.1792,1)--(17.1792,2)--(18.1792,1)--(18.5483,0)--(19.2983,0);
\end{tikzpicture}}
\hspace{-45pt}
\raisebox{60pt}{$\rho_\gamma<\rho_\beta$}
  \caption{The phases of grounded Dyck paths $\GGG^{(\beta,\gamma)}$}\label{figPhasePlane}
\end{figure}

Above this curve and on it, the dominant singularity of the \igf{} is a double pole and the critical exponent equals~$1$; below the curve, the \igf{} has a dominant square root singularity and the critical exponent equals~$-\frac32$. Thus,
\[
\big|\GGG_{\leqs x}\big|
\ssim
\begin{cases}
c_1 
\+ \rho_\beta^{-x} \+ x ,
 & \text{if $\rho_\beta \leqs \rho_\gamma$}, \\[3pt] 
c_2 
 \+ \rho_\gamma^{-x} \+ x^{-3/2} ,
 & \text{if $\rho_\gamma < \rho_\beta$},
\end{cases}
\]
where $c_1$ and $c_2$ depend only on $\beta$ and $\gamma$.


The structure of a typical path also depends on the relationship between $\rho_\beta$ and $\rho_\gamma$.
Given a path $\ppp\in\GGG^{(\beta,\gamma)}$, let $U(\ppp)$ be the number of up-steps (equal to the number of down-steps) in~$\ppp$,
and let $L(\ppp)$ be the number of level steps in~$\ppp$.
Asymptotically, there is a trichotomy:
\[
\begin{array}{lll}
  \expecn{\leqs x}{U} \ssim u_1,          &\;\; \expecn{\leqs x}{L} \ssim \ell_1 x, &\;\; \text{if $\rho_\beta < \rho_\gamma$}, \\[6pt]
  \expecn{\leqs x}{U} \ssim u_2 \sqrt{x}, &\;\; \expecn{\leqs x}{L} \ssim \ell_2 x, &\;\; \text{if $\rho_\beta = \rho_\gamma$}, \\[6pt]
  \expecn{\leqs x}{U} \ssim u_3 x,        &\;\; \expecn{\leqs x}{L} \ssim \ell_3,   &\;\; \text{if $\rho_\beta > \rho_\gamma$}, 
\end{array}
\]
for constants $u_i$ and $\ell_i$ that depend on $\beta$ and $\gamma$.

If $\rho_\beta < \rho_\gamma$, then the expected number of up-steps is constant.
However, on the critical curve, the expected number of up-steps in a path of width at most $x$ grows like $\sqrt{x}$.
In both of these cases, we typically see a large plain containing a small mountain range.
In contrast, if $\rho_\beta > \rho_\gamma$, then the expected number of level steps is constant, and
we typically see a large mountain range rising from a small plain.
See Figure~\ref{figPhasePlane} for an illustration.


\section{Dirichlet series and Ribenboim series}\label{sectTheory}

\subsection{Dirichlet series}

A generalised \emph{Dirichlet series} is a series of the form
\[
F(s) \eq \sum_{i\geqs1} a_i \, e^{-\lambda_i s} ,
\]
where $(\lambda_i)_{i\geqs1}$ is a nonnegative strictly increasing sequence tending to infinity.

The primary reference for standard results on Dirichlet series is Hardy and Riesz~\cite{HR1915}. For a more modern presentation, see McCarthy~\cite{McCarthy2018}.


The region of convergence for a Dirichlet series is a right half-plane:
\begin{propO}[{\cite[Theorem~3]{HR1915}}]\label{propDirichletConvergence}
  For any Dirichlet series $F(s)$, there exists a unique value $\mu\in[0,\infty]$,
  known as its \emph{abscissa of convergence},
  such that
  $F(s)$ is convergent for all $s$ such that $\re s>\mu$, and
  $F(s)$ is not convergent for any $s$ such that $\re s<\mu$.
  The line $\re s=\mu$ is known as \emph{the line of convergence} for $F(s)$.
\end{propO}

A Dirichlet series $F(s)=\sum_{\lambda\in\Lambda} a_\lambda e^{-\lambda s}$ is analytic within its half-plane of convergence, since it has a well-defined derivative, also a Dirichlet series: $F'(s)=\sum_{\lambda\in\Lambda} -a_\lambda \lambda e^{-\lambda s}$.

Moreover, if a Dirichlet series has nonnegative coefficients, then it is singular at the real point on its line of convergence:
\begin{propO}[{\cite[Theorem~10]{HR1915}}]\label{propDirichletSingularity}
  If $F(s)$ is a Dirichlet series with nonnegative coefficients and a finite abscissa of convergence~$\mu$, then $\mu$ is a singularity of~$F$.
\end{propO}


\subsection{Ribenboim series}

\cbstart
In general, functions defined by Ribenboim series are multi-valued on the complex plane.
Specifically, if $z=re^{i\theta}$, then the monomial $z^\lambda$ can take any value in $\big\{ r^\lambda e^{i\lambda(\theta+2k\pi)} : k\in\bbZ \big\}$, depending on which branch is chosen.
If $\lambda$ is irrational, then this set 
is countably infinite, each choice of $k$ giving a distinct value.

Analytic continuation can be used to convert such a multi-valued function 
$f\!:\!\bbC\to\bbC$ into a single-valued function $\widetilde{f}\!:\!S\to\bbC$ by lifting its domain onto a surface $S$, known as a \emph{Riemann surface}, that covers the punctured complex plane in an infinite-to-one way.
For an introduction to Riemann surfaces and analytic continuation, see \cite[Chapter~7]{Wegert2012} or \cite[Section~14.5]{ST2018}.

For Ribenboim series, we use the covering of $\bbC \setminus\{0\}$ by the helicoidal Riemann surface
\[ 
\helix \defeq \big\{ (re^{i\theta},\theta) \::\: r>0,\,\, \theta\in\bbR \big\} ,
\]
with the natural projection $(z,\theta)\mapsto z$ as the covering map.
See Figure~\ref{figRiemannSurface} for an illustration.
\begin{figure}[t]
  \centering
  \includegraphics[width=1.6in]{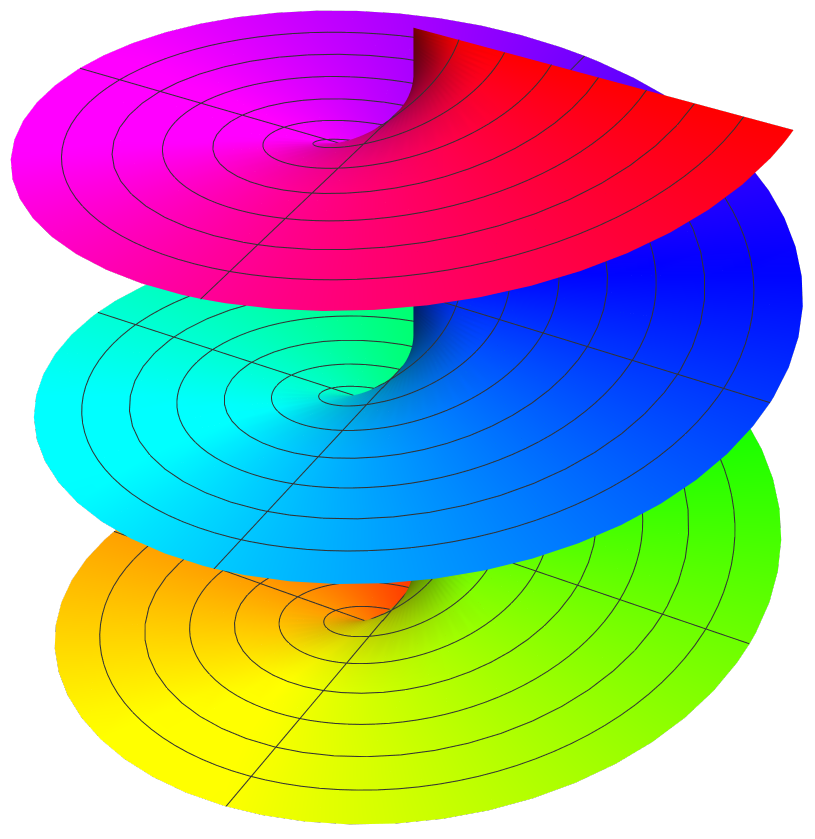}
  \caption{An illustration of (part of a disk in) the Riemann surface $\helix$}\label{figRiemannSurface}
\end{figure}

\together4
If $f(z)=\sum_{\lambda\in\Lambda}c_\lambda z^\lambda$, then
\[
\widetilde{f}(re^{i\theta},\theta) \eq \sum_{\lambda\in\Lambda} c_\lambda r^\lambda e^{i\lambda\theta} 
\]
is a single-valued function from $\helix$ to $\bbC$
such that $\widetilde{f}(z,\arg z)=f(z)$.

The exponential transform of $f$ then satisfies 
\[
\widetilde{f}(re^{i\theta},\theta) \eq F(-\log r-i\theta)
\qquad
\text{and}
\qquad
F(s) \eq \widetilde{f}\big(e^{-s},-\im s\big) , 
\]
for every point in $\helix$ and every $s\in\bbC$.
The exponential transform thus bijectively ``unwraps'' $\helix$, stretching the origin out to $+\infty$.

The properties of Dirichlet series stated above are sufficient to establish convergence and analyticity in Ribenboim series:
\cbend


\begin{propO}\label{propRibenboimConvergence}
    If $f(z)$ is a Ribenboim series, then there exists some $\rho\in[0,\infty]$ such that $f(z)$ is convergent for all $z\in\bbC$ such that $|z|<\rho$, and is not convergent when $|z|>\rho$. Moreover, the function $f$ is analytic in the open cut disk $D(0,\rho)\setminus\mathbb{R}^{\leqs0}$
    \cbstart
    and extends analytically to the open disk $\{ (z, \theta) \in \helix : |z|<\rho \}$ in the Riemann surface.
    \cbend
    If $f(z)$ has nonnegative coefficients and $\rho$ is finite, then $f$ has a singularity at $z=\rho$.
\end{propO}
\begin{proof}
    The exponential transform $F$ of $f$ is a Dirichlet series, satisfying
    $F(s)=f(e^{-s})$ when $\im s\in(-\pi,\pi]$ and $f(z)=F(\log(1/z))$ if $z\neq0$.
    \cbstart
    Moreover, if $\widetilde{f}:\helix\to\bbC$ is defined by
    $\widetilde{f}(re^{i\theta},\theta) \eq F(-\log r-i\theta)$,
    then we have $\widetilde{f}(z,\arg z)=f(z)$.
    \cbend
    
    The results then follow from Propositions~\ref{propDirichletConvergence} and~\ref{propDirichletSingularity}
    and the fact that the composition of two analytic functions is analytic.
\end{proof}

\subsection{Proving Theorem~\ref{thmMain}}

Given a Dirichlet series $\sum_{\lambda\in\Lambda} a_\lambda e^{-\lambda s}$, its
\emph{summatory function} is 
the function $A(x)=\sum_{\lambda\leqs x} a_\lambda$.
A Dirichlet series can be recovered from the Laplace transform of its summatory function as follows:
\begin{obsO}\label{obsLaplaceTranform}
\[
    F(s)
    \eq \sum_{\lambda\in\Lambda} a_\lambda e^{-\lambda s}
    \eq \sum_{\lambda\in\Lambda} a_\lambda\+ s\int_\lambda ^\infty e^{-sx} \+dx
    \eq s\int_0^\infty A(x)\+e^{-sx} \+dx ,
\]
where $A(x)=\sum_{\lambda\leqs x} a_\lambda$.
\end{obsO}

To prove our main theorem, we combine this observation with an appropriate Wiener--Ikehara Tauberian theorem.
The following variant is due to Ryo Kato, which extends an earlier theorem of Anthony Kable~\cite[Theorem~1]{Kable2008}.


\begin{propO}[{Kato~\cite[Theorem~2.1]{Kato2015}}]\label{propTauberian}
  Let $j:\bbR\to[0,\infty)$ be a non-decreasing function whose Laplace transform
   \[
   J(s) \eq \int_0^\infty j(x)\+e^{-sx}\+dx
   \]
   converges absolutely when $\re s>1$.
   Given $\alpha\geqs1$ and $L,\, \alpha_1, \ldots ,\alpha_r >0$, with $\alpha>\alpha_i$ for each~$i$, and $L_1,\ldots,L_r\in\bbC$, let
   \[
   K(s) \eq J(s) \sminus \frac{L}{(s-1)^\alpha} \sminus \sum_{i=1}^r \frac{L_i}{(s-1)^{\alpha_i}}.
   \]
   If $K$ extends continuously to the closed half-plane $\chplane1$, then
   \[
   j(x) \ssim \frac{L}{\Gamma(\alpha)}\+e^{x}\+x^{\alpha-1} .
   \]
\end{propO}

The following proposition establishes the asymptotics of the summatory function of suitably well-behaved Dirichlet series.
Subsequently, we transfer this result to Ribenboim series.

\begin{propO}\label{propDirichletAsymptotics}
Suppose $F(s)=\sum_{\lambda\in\Lambda} a_\lambda e^{-\lambda s}$ is a Dirichlet series with nonnegative coefficients and abscissa of convergence $\mu>0$.
Suppose that $\mu$ is the unique singularity of $F$ on the line $\mu+i\bbR$, and that one can write
\[
    F(s) \eq G(s)\splus\frac{H(s)}{(s-\mu)^\alpha} ,
\]
where $\alpha\notin-\bbN$, and both $G$ and $H$ are analytic on the closed half-plane $\chplane\mu$, then the asymptotics of the summatory function of $F$ are given by
\[
    \sum_{\lambda\leqs x} a_\lambda
    \ssim \frac{H(\mu)}{\mu\,\Gamma(\alpha)} \+ e^{\mu x} \+ x^{\alpha-1}.
\]
\end{propO}
\newcommand{\til}{\widetilde}
\begin{proof}
  Let $A(x)=\sum_{\lambda\leqs x} a_\lambda$ be the summatory function of~$F$, and let $L=H(\mu)$.

  We re-scale so that the abscissa of convergence is at 1.
  By Observation~\ref{obsLaplaceTranform}, when $\re s>1$, we have
  \begin{equation}\label{eqFTildeLaplace}
  \til{F}(s) \defeq
  \frac{F(\mu s)}s
    \eq \int_0^\infty A\!\left(\frac{x}{\mu}\right) e^{-sx} \+dx ,
  \end{equation}
  where $A(x/\mu)$ is nonnegative and non-decreasing.

  Moreover,
  \begin{equation}\label{eqFGHTilde}
    \til{F}(s) \eq \til{G}(s)\splus\frac{\til{H}(s)}{(s-1)^\alpha} ,
  \end{equation}
  for functions $\til{G}$ and $\til{H}$ analytic on the closed half-plane $\chplane1$.

  We consider two cases. Assume first that $\alpha\geqs1$.
  If $r$ is the greatest integer less than $\alpha$, then by taking the Taylor expansion about $s=1$, we have
  \[
  \til{F}(s) \eq K(s) \splus \frac{\til{H}(1)}{(s-1)^\alpha} \splus \sum_{i=1}^r \frac{L_i}{(s-1)^{\alpha-i}} ,
  \]
  for some constants $L_1,\ldots,L_r$.
  Since the remaining terms of the Taylor expansion are continuous at $s=1$, the function
  $K$ extends continuously to the closed half-plane $\chplane1$.
  
  Hence
  $A(x/\mu)$ satisfies the conditions of the Tauberian theorem (Proposition~\ref{propTauberian}).
  Therefore,
  \[
  A(x/\mu) \ssim \frac{\til{H}(1)}{\Gamma(\alpha)}\+e^x\+ x^{\alpha-1} .
  \]
  Thus, since $\til{H}(1)=L/\mu^\alpha$, we have
  \[
  A(x) \ssim \frac{L}{\mu\,\Gamma(\alpha)} \+ e^{\mu x} \+ x^{\alpha-1} ,
  \]
  as required.

  Now suppose that $\alpha<1$.
  Since $\til{F}(s)$ is analytic, so is the Laplace transform on the right of equation~\eqref{eqFTildeLaplace},
  and we can legitimately differentiate under the integral sign.
  For each $n\in\bbN$, we have
  \[
    (-1)^n\,\frac{\mathrm{d}^n}{\mathrm{d}s^n}\til{F}(s)
    \eq
    \int_0^\infty x^n\+A\!\left(\frac{x}{\mu}\right) e^{-sx} \+dx .
  \]
  Taking the $n$th derivative in~\eqref{eqFGHTilde}, we can also write
  \[
    (-1)^n\,\frac{\mathrm{d}^n}{\mathrm{d}s^n}\til{F}(s)
    \eq
    Y(s) \splus \frac{Z(s)}{(s-1)^{\alpha+n}}
  \]
  where $Y$ and $Z$ are both analytic on the closed half-plane $\chplane1$.
  Moreover, if $1$ is the only singularity of $\til{F}$ on the line $1+i\mathbb{R}$, then this is also the case for its derivatives.

  If we let $n=\ceil{1-\alpha}$ so that $\alpha+n\geqs1$, then the Tauberian theorem now applies to $x^nA(x/\mu)$, so we have
  \[
  x^n A(x/\mu) \ssim \frac{Z(1)}{\Gamma(\alpha+n)}\+e^x\+ x^{\alpha+n-1} .
  \]
  Now, since $\alpha\notin-\bbN$,
  \[
  Z(1) \eq \til{H}(1)\+\alpha\+(\alpha+1)\+\cdots\+(\alpha+n-1) \eq \frac{L\,\Gamma(\alpha+n)}{\mu^\alpha\+\Gamma(\alpha)} .
  \]
  Hence,
  \[
  A(x) \ssim \frac{L}{\mu\,\Gamma(\alpha)} \+ e^{\mu x} \+ x^{\alpha-1}
  \]
  when $\alpha<1$ too.
\end{proof}

This is sufficient for the proof of our main theorem,
which we restate first.
\cbstart
Note that the condition placed on the exponential transform of $g$ is equivalent to the requirement that $g$
extends analytically 
to the closed disk $\{ (z, \theta) \in \helix : |z|\leqs\rho \}$ in the Riemann surface.
\cbend

\textbf{Theorem \ref{thmMain}.} \emph{\mainTheoremText}

\begin{proof}
\cbstart
Suppose the conditions of the theorem are met and that $F_\CCC$ is the \dgf{} of $\CCC$.
Then, since $\CCC$ is primitive, $\mu=\log(1/\rho)$ is the unique singularity of $F_\CCC$ on the line $\mu+i\bbR$.
Thus $F_\CCC(s)$ is analytic on the closed half-plane $\chplane{\mu}$ except at $s=\mu$.

Moreover, if $G(s)$ is the exponential transform of $g(z)$, then $F_\CCC(s)-G(s)$ is also analytic on $\chplane{\mu}\setminus\{\mu\}$.

Let $H(s)=(s-\mu)^\alpha\big(F_\CCC(s) - G(s)\big)$.
Then the only point in the closed half-plane $\chplane{\mu}$ at which $H(s)$ could possibly have a singularity is $s=\mu$.

Now, on the principal strip, $-\pi<\im s\leqs\pi$, we have
\[
F_\CCC(s) \sminus G(s) \eq \frac{h(e^{-s})}{(1-e^{\mu-s})^\alpha} ,
\]
where $h(e^{-s})$ is analytic on this strip when $\re s\geqs\mu$ (including at $s=\mu$).

Thus, on the principal strip,
\[
H(s) \eq \left(\frac{s-\mu}{1-e^{\mu-s}}\right)^{\!\alpha} h(e^{-s}) .
\]
Now, $\left(\frac{s-\mu}{1-e^{\mu-s}}\right)^{\!\alpha}$ is analytic at $s=\mu$, having a removable singularity there. 
So, in fact, $H(s)$ has no singularity at~$s=\mu$.

So,
  \[
  F_\CCC(s) \eq G(s)\splus\frac{H(s)}{(s-\mu)^\alpha},
  \]
where both $G$ and $H$ are analytic on the closed half-plane $\chplane\mu$.
\cbend

Hence, by Proposition~\ref{propDirichletAsymptotics},
  \[
  |\CCC_{\leqs x}|
  \ssim \frac{H(\mu)}{\log(1/\rho)\,\Gamma(\alpha)} \+ \rho^{- x} \+ x^{\alpha-1}.
  \]
  It only remains to check that $H(\mu)=h(\rho)$.
  Now, 
  \[
  \lim_{s\to\mu} \,\frac{s-\mu}{1-e^{\mu-s}} \eq 1 .
  \]
  Thus $H(\mu)=h(e^{-\mu})=h(\rho)$ as required.
\end{proof}

{
\cbstart
We believe that the conclusion of Theorem~\ref{thmMain} should follow even if the function $g$ is only analytic on the closed cut disk in~$\bbC$.
We conclude with this conjecture:

\begin{conjO}\label{conjAnalytic}
  Suppose $f(z)$ is a Ribenboim series with nonnegative coefficients and finite radius of convergence $\rho$, and one can write
  \[
  f(z) \eq g(z) \splus \frac{h(z)}{(1-z/\rho)^\alpha} ,
  \]
  where $\alpha\notin-\bbN$, 
  and both $g$ and $h$ are analytic on the cut disk $\overline{D}(0,\rho) \setminus \bbR^{\leqs0}$.
  
  Then there exist functions $g_1$ and $h_1$, both analytic on $\overline{D}(0,\rho) \setminus \bbR^{\leqs0}$,
  such that $h_1(\rho)=h(\rho)$ and
  the exponential transform of $g_1$ is analytic on 
  $\chplane{\log(1/\rho)}$,
  and
  \[
  f(z) \eq g_1(z) \splus \frac{h_1(z)}{(1-z/\rho)^\alpha} .
  \]
\end{conjO}
\cbend
} 

\bibliographystyle{plain}
{\footnotesize\bibliography{../bib/mybib}}

\end{document}